\documentclass[a4paper,12pt,reqno]{amsart}
\usepackage{amsfonts}
\usepackage{amsmath,array}
\usepackage{amssymb}
\usepackage{mathrsfs}
\usepackage[colorlinks]{hyperref}
 \usepackage{graphicx}
\usepackage{enumerate}
\usepackage[usenames]{color}
\usepackage{comment}
\makeatletter
\newcommand{\setword}[2]{%
  \phantomsection
  #1\def\@currentlabel{\unexpanded{#1}}\label{#2}%
}
\makeatother
\newtheorem{thm}{Theorem}[section]

\newtheorem{lem}[thm]{Lemma}

\newtheorem{remark}[thm]{Remark}
\numberwithin{equation}{section}
\usepackage[english]{babel} 
\usepackage{blindtext}

\theoremstyle{definition}
\newtheorem{definition}[thm]{Definition}

\begin{document}

\allowdisplaybreaks 
 \title[Mixed local-nonlocal quasilinear problems]{Mixed local-nonlocal quasilinear problems with mixed interpolated Hardy potential}

\author[Yergen Aikyn, Sekhar Ghosh, Vishvesh Kumar, and Michael Ruzhansky]{Yergen Aikyn, Sekhar Ghosh, Vishvesh Kumar, and Michael Ruzhansky}

\address[Yergen Aikyn]{Department of Mathematics: Analysis, Logic and Discrete Mathematics, Ghent University, Ghent, Belgium}
\email{aikynyergen@gmail.com}

\address[ Sekhar Ghosh]{Department of Mathematics, National Institute of Technology Calicut, Kozhikode, Kerala, India - 673601}
\email{sekharghosh1234@gmail.com / sekharghosh@nitc.ac.in}

\address[Vishvesh Kumar]{Department of Mathematical Sciences,
		Indian Institute of Technology (BHU),
		Varanasi, Uttar Pradesh, 221005, India.}
\email{vishveshmishra@gmail.com }

\address[Michael Ruzhansky]{Department of Mathematics: Analysis, Logic and Discrete Mathematics, Ghent University, Ghent, Belgium\newline and \newline
School of Mathematical Sciences, Queen Marry University of London, United Kingdom}
\email{michael.ruzhansky@ugent.be}
\date{}

\begin{abstract}
This paper addresses the existence of nontrivial solutions to a class of mixed local-nonlocal problems involving a mixed interpolated Hardy potential. We first establish a concentration-compactness principle for mixed local and nonlocal operators. This result is combined with Ricceri's variational principle to obtain an existence result for quasilinear elliptic problems under different growth assumptions on the nonlinearity. Furthermore, we apply the classical mountain pass theorem to obtain a second existence result in the superlinear case.
\end{abstract}


\maketitle

\section{Introduction and main results}

Recently, significant attention has been devoted to nonlinear problems driven by mixed local--nonlocal operators. These models arise naturally in a variety of applications, including population dynamics, finance, and control theory, where both diffusion effects and nonlocal interactions must be taken into account. From a mathematical viewpoint, the study of such operators is delicate due to the lack of scaling invariance and the combination of local and nonlocal features. We refer to \cite{BC-2013A2, DLV-2023A1, DV-2021A1} and the references therein for a comprehensive overview of these developments and their applications.

In this regard, substantial work has been done on the study of the problem
\begin{equation*}
    \begin{cases}-\Delta_{p} +(-\Delta_p)^{s} u=f(x, u) & \text { in } \Omega,  \\ u=0 & \text { on } \mathbb{R}^{N} \backslash \Omega ,\end{cases}
\end{equation*}
where $p> 1$, $\Omega$ is a smooth bounded domain in $\mathbb{R}^N$, and $f(x, u)$ is a nonlinearity satisfying  different types of growth conditions. Here $-\Delta_{p}$ is the $p$-Laplace operator given by
$$-\Delta_{p} u:=-\operatorname{div}\left(|\nabla u|^{p-2} \nabla u\right)$$ 
and $\left(-\Delta_p\right)^{s}$ is its fractional counterpart which is defined as
\begin{equation*}
    \left(-\Delta_p\right)^{s} u(x):=2 c_{N,p,s} \lim _{\epsilon \rightarrow 0} \int_{\mathbb{R}^{N} \backslash B_{\varepsilon}(x)} \frac{|u(x)-u(y)|^{p-2}(u(x)-u(y))}{|x-y|^{N+p s}} d y \quad \text{ for all } x \in \mathbb{R}^{N},
\end{equation*}
for any $u \in C_{0}^{\infty}\left(\mathbb{R}^{N}\right)$.
 We refer the interested reader, among other works, to the papers \cite{BDVV-2025, DFZ-2024, GK-2022, GL-2023, SFV-2024} and to the references therein. 

Compared with the classical mixed local-nonlocal case, the study of such operators in the presence of the Hardy potential is still underdeveloped. The first work on this topic is due to Biagi et al. \cite{BEMV-2025}, where the authors investigated the existence, uniqueness, and optimal summability of solutions to the following problem
\begin{equation}\label{Biagi pr}
    \begin{cases} -\Delta u+\left(-\Delta\right)^{s} u-\gamma \frac{u}{|x|^{2 }}=f(x) & \text { in } \Omega,  \\ u=0 & \text { on } \mathbb{R}^{N} \backslash \Omega .\end{cases}
\end{equation}
Continuing this line of research, Malhotra et al. \cite{MGS2025ArX} studied the Brezis-Nirenberg type problem associated with the operator given in \eqref{Biagi pr}. A primary reason for considering this operator is that the minimizing value of the following minimization problem
\begin{equation*}
    \Gamma(\Omega):=\inf \left\{\|\nabla u\|_{L^2\left(\mathbb{R}^N\right)}^2+\frac{c_{s,N}}{2}[u]_s^2: u \in C_0^{\infty}(\Omega) \text { with } \int_{\mathbb{R}^N} \frac{u^2}{|x|^2} d x=1\right\}
\end{equation*}
coincides with the optimal local Hardy constant. Recently, Malhotra et al. \cite{MGS2025} studied the eigenvalue problem driven by the mixed local and nonlocal $p$-Laplacian operator involving the interpolated Hardy potential. More precisely, they considered the following operator 
\begin{equation}\label{main operator hardy}
    \mathcal{T}(u):=-\Delta_{p} u+\left(-\Delta_p\right)^{s} u-\mu \frac{u^{p-1}}{|x|^{p \theta}},
\end{equation}
where $0<s<1<p<N$, $\theta \in[s, 1]$ with $s\in(0,1)$. 
The Hardy potential appearing in the definition of the operator $\mathcal{T}$ is motivated by the interaction between the classical and nonlocal Hardy inequalities. A key feature of this approach is that it leads to an interpolated Hardy inequality established in \cite{MGS2025}. In this context, the following constant $\bar{\mu}(\theta)$ arises, which is defined by
\begin{equation}\label{best Hardy constant}
\bar{\mu}(\theta):= \begin{cases}C_{N, p, s} & \text { if } \theta=s, \\ \min \left\{\frac{C_{H}(1-s)}{\theta-s}, \frac{C_{N, p, s}(1-s)}{1-\theta}\right\} & \text { if } \theta \in(s, 1), \\ C_{H} & \text { if } \theta=1,\end{cases}
\end{equation}
where $C_{H}$ and $C_{N, p, s}$ are the best constants in classical and fractional Hardy inequalities, respectively.

Motivated by the above works, we investigate existence results for certain Dirichlet problems involving the operator $\mathcal{T}$ as in \eqref{main operator hardy}. As usual, the main difficulty in problems involving the Hardy potential is the lack of compactness of the Hardy embedding. This poses several difficulties for the application of variational methods. The standard tool to handle this lack of compactness is the concentration-compactness principle of Lions \cite{Lions-1985}, which is based on the study of weakly convergent sequences in measure spaces.

We prove the following concentration-compactness principle for mixed local-nonlocal operators.
\begin{thm}\label{CCP thm}
Let $\Omega$ be an open bounded subset of $\mathbb{R}^{N}$ with $0\in \Omega$ and $X(\Omega)$ be the Sobolev space as defined in Section \ref{sec2}. Let $\{u_{k}\}$ be a sequence in $X(\Omega)$ such that $u_k\rightharpoonup u$ weakly for some $u\in X(\Omega)$. Then there exist two finite measures $\omega$ and $\sigma$ in $\mathbb{R}^{N}$ such that
\begin{align}\label{ccp pr 1}
    \begin{split}
        &\left|\nabla u_{k}(x)\right|^{p} d x +\left(\int_{\mathbb{R}^N}\frac{|u_k(x)-u_k(y)|^p}{|x-y|^{N+ps}}dy\right)dx\stackrel{*}{\rightharpoonup} \omega \quad \text { and }\\
        &\frac{\left|u_{k}(x)\right|^{p}}{|x|^{p\theta}} dx\stackrel{*}{\rightharpoonup} \sigma \quad \text { in } \mathcal{M}\left(\mathbb{R}^{N}\right). 
    \end{split}
\end{align}
Furthermore, there exist two nonnegative numbers $\omega_{0}, \sigma_{0}$ such that
\begin{equation}\label{ccp pr 2}
\sigma= \frac{|u(x)|^{p}}{|x|^{p\theta}}dx+\sigma_{0} \delta_{0} 
\end{equation}
and
\begin{equation}\label{ccp pr 3}
\omega \geq\left|\nabla u(x)\right|^{p} d x+\left(\int_{\mathbb{R}^N}\frac{|u(x)-u(y)|^p}{|x-y|^{N+ps}}dy\right)dx+\omega_{0} \delta_{0}, \quad 0 \leq \bar{\mu}(\theta) \sigma_{0} \leq \omega_{0}, 
\end{equation}
where $\bar{\mu}(\theta)$ is the Hardy constant defined in \eqref{best Hardy constant}.
\end{thm}

Theorem \ref{CCP thm} can be established by adapting the method used in \cite[Theorem 1.1]{F-P} (see also
\cite[Lemma 3.1]{Montefusco}). 

As an immediate consequence of Theorem \ref{CCP thm}, we establish that the functional
\begin{equation}\label{wlsc func1}
    \mathcal{I}_{\mu}(u):=\frac{1}{p}\left(\int_{\Omega}\left|\nabla u\right|^p dx+\iint_{\mathbb{R}^{2N}} \frac{|u(x)-u(y)|^{p}}{|x-y|^{N+sp}} dxdy-\mu \int_{\Omega} \frac{\left|u\right|^p}{|x|^{p\theta}} dx\right)
\end{equation}
is weakly lower semicontinuous and coercive in $X(\Omega)$, provided that $\mu<\bar{\mu}(\theta)$. We note that this method of proving weak lower semicontinuity of functionals using the concentration–compactness principle goes back to Montefusco \cite{Montefusco}.

Next, we consider the following nonlinear problem 
\begin{equation} \label{main problem 0}
    \begin{cases}-\Delta_{p}u+(-\Delta_p)^{s} u=\mu \frac{|u|^{p-2}u}{|x|^{p \theta}}+\lambda f(x,u) & \text { in } \Omega, \\ u=0 & \text { in } \mathbb{R}^{N} \backslash \Omega,\end{cases}
\end{equation}
where $\Omega \subset \mathbb{R}^{N}$ is a bounded domain containing the origin with smooth boundary, $\lambda>0$, $0<s<1<p<N,$ $\theta \in[s, 1]$, $\mu \in (0, \bar{\mu}(\theta))$, and $f: \Omega \times \mathbb{R} \rightarrow \mathbb{R}$ is a Carath\'eodory function (that is, $f$ is measurable in $x\in\Omega$ and continuous in $t\in\mathbb{R})$, satisfying the following condition
\begin{equation}\label{cond on f}
|f(x, t)| \leq a_{1}+a_{2}|t|^{q-1}, \quad \forall(x, t) \in \Omega \times \mathbb{R},
\end{equation}
where $a_{1}, a_{2}$ are non-negative constants and $q \in (1, p N /(N-p))$.  

The main result concerning problem \eqref{main problem 0} reads as follows.
\begin{thm}\label{main result thm}
   Assume that $f: \Omega \times \mathbb{R} \rightarrow \mathbb{R}$ is a Carath\'eodory function satisfying $f(x, 0) \neq 0$ in $\Omega$ and condition \eqref{cond on f}. Then, for every $\mu \in(0, \bar{\mu}(\theta))$ there exists a constant $\Lambda>0$ defined by
\begin{equation}\label{value of lambdamu}
\Lambda:=q \sup _{\sigma>0}\left(\frac{\sigma^{p-1}}{q a_{1} C_{1}\left(\frac{p \bar{\mu}(\theta)}{\bar{\mu}(\theta)-\mu}\right)^{1 / p}+a_{2} C_{q}^{q}\left(\frac{p \bar{\mu}(\theta)}{\bar{\mu}(\theta)-\mu}\right)^{q / p} \sigma^{q-1}}\right),
\end{equation}
such that problem \eqref{main problem 0} admits at least one nontrivial weak solution $u_{\lambda} \in  X(\Omega)$ for every $\lambda \in( 0, \Lambda)$, where $C_r>0$ denotes the optimal embedding constant in $X(\Omega)\hookrightarrow L^{r}(\Omega)$ for $1\leq r\leq q$. Moreover, we have
\begin{equation*}
\lim _{\lambda \rightarrow 0^{+}}\left\|u_{\lambda}\right\|_{X(\Omega)}=0
\end{equation*}
and the map $\lambda \mapsto \mathcal{J}_{\lambda, \mu}\left(u_{\lambda}\right)$ is negative and strictly decreasing in $(0, \Lambda)$, where $\mathcal{J}_{\lambda, \mu}$ is defined as in \eqref{def j}.
\end{thm}

The method we use to establish Theorem \ref{main result thm} is similar to the approach of Ferrara and Bisci \cite{FB-2014} and is based on the weak lower semicontinuity of the functional given in \eqref{wlsc func1}. We also mention that a fractional counterpart of this argument was developed in \cite{AB-2023}.

The last part of this work is devoted to the study of the particular case of problem \eqref{main problem 0} when $f(x,u):=|u|^{r-2}u$ and $p<r<p^{*}$. More precisely, we consider the problem
\begin{equation} \label{main problem 0.1}
    \begin{cases}-\Delta_{p}u+(-\Delta_p)^{s} u=\mu \frac{|u|^{p-2}u}{|x|^{p \theta}}+\lambda|u|^{r-2}u & \text { in } \Omega, \\ u=0 & \text { in } \mathbb{R}^{N} \backslash \Omega,\end{cases}
\end{equation}
where $\Omega \subset \mathbb{R}^{N}$ is a bounded domain containing the origin with smooth boundary, $\lambda>0$, $0<s<1<p<N,$ $\theta \in[s, 1]$, $\mu \in (0, \bar{\mu}(\theta))$, and $p<r<p^{*}$. We prove the existence of a mountain pass type solution to problem \eqref{main problem 0.1}. The main difficulty in establishing this result lies in proving the convergence of the gradients, which is required to verify the Palais--Smale condition. To overcome this difficulty, we employ a technique introduced by Boccardo and Murat \cite{BM-1992}. The last main result of this paper is stated below.
\begin{thm}\label{main result thm 1}
  Let $\lambda>0$, $\mu \in (0,\bar{\mu}(\theta))$ and $p<r<p^{*}$. Then the problem \eqref{main problem 0.1}  admits a nontrivial mountain pass solution.
\end{thm}

The paper is organized as follows. In Section \ref{sec2}, we state the functional setting and preliminary results used throughout the paper. Section \ref{sec3} contains the proof of Theorem \ref{CCP thm}. In Sections \ref{sec4} and \ref{sec5}, we prove the main existence results stated in Theorems \ref{main result thm} and \ref{main result thm 1}, respectively.

\section{Preliminaries} \label{sec2}
We begin this section by introducing some notation and function spaces used throughout the paper, and by stating some basic results that will be needed later.

Let $\Omega \subset \mathbb{R}^{N}$ (with $N \geq 3$) be a bounded domain containing the origin and with smooth boundary $\partial \Omega$. Let $p>1$ and $s \in(0,1)$ be real numbers such that $0<s<1<p<N$. For any  measurable function $u: \mathbb{R}^N \rightarrow \mathbb{R},$ consider the Gagliardo seminorm of $u$ defined by
\begin{equation*}
   [u]_{s, p}=\left(\iint_{\mathbb{R}^{2 N}} \frac{|u(x)-u(y)|^{p}}{|x-y|^{N+p s}} d x d y\right)^{1 / p}.
\end{equation*}

We define the function space $X(\Omega)$ as the completion of $C_0^{\infty}(\Omega)$ with respect to the following norm 
\begin{equation}\label{global norm}
    \|u\|_{X}=\left(\|\nabla u\|_{p}^{p}+[u]_{s, p}^{p}\right)^{1 / p}.
\end{equation}

It is well-known that $X(\Omega)$ is a uniformly convex Banach space. Recall that by continuous embedding of $W^{1,p}\left(\mathbb{R}^n\right)$ into $W^{s,p}\left(\mathbb{R}^n\right)$ (see, e.g., \cite[Proposition 2.2]{NPV-2012}), there exists a positive constant $C>0$ such that
\begin{equation}\label{Gagl ineq mixed}
[u]_{s,p}^p \leq C\|u\|_{W^{1,p}\left(\mathbb{R}^n\right)}^p=C\left(\|u\|_{p}^p+\|\nabla u\|_{p}^p\right) \quad \text { for all } u \in C_{0}^{\infty}(\Omega) .
\end{equation}
Then, applying \eqref{Gagl ineq mixed} and the classical Poincaré inequality, we conclude that there exist constants $c_1,c_2>0$ such that
\begin{equation*}
c_1\|u\|_{W^{1,p}\left(\mathbb{R}^N\right)} \leq\|u\|_{X} \leq c_2\|u\|_{W^{1,p}\left(\mathbb{R}^N\right)} \quad \text { for all } u \in C_{0}^{\infty}(\Omega) .
\end{equation*}
This shows that the norm $\|\cdot\|_{W^{1, p}(\mathbb{R}^N)}$ is equivalent to $\|\cdot\|_{X}$ on $C_{0}^{\infty}(\Omega)$, and therefore, we have the following characterization of $X(\Omega)$
\begin{align*}
X(\Omega) & =\overline{C_0^{\infty}(\Omega)}^{\|\cdot\|_{W^{1,p}\left(\mathbb{R}^N\right)}}=\\
&=\left\{u \in W^{1,p}\left(\mathbb{R}^N\right):\left.u\right|_{\Omega} \in W^{1,p}_0(\Omega) \text { and } u \equiv 0 \text { a.e. in } \mathbb{R}^N \backslash \Omega\right\} .
\end{align*}

Using the classical Sobolev inequality, we obtain the following Sobolev inequality for the norm defined in \eqref{global norm}:
\begin{equation}\label{mixed sobol ineq}
\|u\|_{L^{p^{*}}(\Omega)}=\|u\|_{L^{p^{*}}\left(\mathbb{R}^{N}\right)} \leq S\|\nabla u\|_{L^{p}\left(\mathbb{R}^{N}\right)} \leq S\|u\|_{X} \quad \text { for all } u \in X(\Omega), 
\end{equation}
where $S>0$ is the best Sobolev constant. Moreover, since $\Omega$ is bounded, we can apply H\"older's inequality in \eqref{mixed sobol ineq} to obtain 
\begin{equation}\label{Sobolev emb}
\|u\|_{L^{r}(\Omega)} \leqslant C_{r}\|u\|_{X}
\end{equation}
for all $r\in[1,p^{*}]$, $u \in X(\Omega)$, and a positive constant $C_{r}>0$. Hence, the embedding
\begin{equation*}
X(\Omega) \hookrightarrow L^{r}\left(\Omega\right)
\end{equation*}
is continuous for $1 \leq r\leq p^{*}$. 
In particular, from \cite[Theorem 2.80]{Demdem-2012}, we conclude that the embedding
\begin{equation}\label{Rel-Kon comp emb}
X(\Omega) \hookrightarrow L^{r}(\Omega) 
\end{equation}
is compact for $1 \leq r<p^{*}$.

Recall the classical Hardy  inequality given by
\begin{equation}\label{local Hardy}
C_{H} \int_{\Omega} \frac{|u|^{p}}{|x|^{p}} d x \leq \int_{\Omega}|\nabla u|^{p} d x 
\end{equation}
for all $1<p<N$ and $u \in W_{0}^{1, p}(\Omega)$ where the constant $C_{H}=\left(\frac{N-p}{p}\right)^{p}$ 
is optimal and not achieved \cite{AP-1998, HLP}. For a comprehensive advancement of Hardy inequalities, we refer to \cite{OK, RuzSur}. A nonlocal version of the Hardy inequality, obtained in \cite{FS-2008}, is stated as follows
\begin{equation}\label{fractional Hardy}
C_{N, p, s} \int_{\mathbb{R}^{N}} \frac{|u|^{p}}{|x|^{p s}} d x \leq \iint_{\mathbb{R}^{2N}}  \frac{|u(x)-u(y)|^{p}}{|x-y|^{N+s p}} d y d x \text { for all } u \in W^{s,p}\left(\mathbb{R}^{N}\right),
\end{equation}
where $N \geq 1$, $s \in(0,1),$ $1<p<N / s$. Here, the optimal constant is defined as
\begin{equation*}
    C_{N, p, s}:=2 \int_{0}^{1} t^{p s-1}\left|1-t^{(N-p s) / p}\right|^{p} \Phi_{N, s, p}(t) d t,
\end{equation*}
where
\begin{equation*}
    \Phi_{N, s, p}(t):= \begin{cases}\left|\mathbb{S}^{N-2}\right| \int_{-1}^{1} \frac{\left(1-r^{2}\right)^{(N-3) / 2} d r}{\left(1-2 r t+t^{2}\right)^{(N+p s) / 2}}, & N \geq 2, \\ \left(\frac{1}{(1-t)^{1+p s}}+\frac{1}{(1+t)^{1+p s}}\right), & N=1.\end{cases}
\end{equation*}

The following mixed interpolated Hardy inequality was established in \cite{MGS2025}.

\begin{lem}
    (Mixed Interpolated Hardy Inequality). Let $\theta \in[s, 1]$ with $s \in(0,1)$ and $1<$ $p<N$. Then for all $u \in X(\Omega)$, we have the following inequality
\begin{equation*}
\int_{\Omega} \frac{|u|^{p}}{|x|^{p \theta}} d x \leq \frac{(\theta-s)}{(1-s) C_{H}} \int_{\Omega}|\nabla u|^{p} d x+\frac{(1-\theta)}{(1-s) C_{N, p, s}}  \iint_{\mathbb{R}^{2N}} \frac{|u(x)-u(y)|^{p}}{|x-y|^{N+s p}} d y d x.
\end{equation*}
\end{lem} 

\begin{remark}
    Taking into account the definition of $\bar{\mu}(\theta)$ in \eqref{best Hardy constant}, we can write the mixed interpolated Hardy inequality as
\begin{equation}\label{mixed Hardy use}
\int_{\Omega} \frac{|u|^{p}}{|x|^{p \theta}} d x \leq \frac{1}{\bar{\mu}(\theta)} \left(\int_{\Omega}|\nabla u|^{p} d x+\iint_{\mathbb{R}^{2N}} \frac{|u(x)-u(y)|^{p}}{|x-y|^{N+s p}} d x d y\right)
\end{equation}
for all $u \in X(\Omega)$.
\end{remark}

We conclude this section by recalling the following results, which will be used to establish the existence results of this paper. 

\begin{thm}\label{abstr CPT}\cite[Theorem 2.5.]{Ricceri-2000}
Let $X$ be a reflexive real Banach space, and let $\Phi, \Psi: X \rightarrow \mathbb{R}$ be two G\^ateaux differentiable functionals on $X$.
Assume that the functional $\Psi$ is sequentially weakly upper semicontinuous and that the functional $\Phi$ is strongly continuous, sequentially weakly lower semicontinuous, and coercive. 

Define
\begin{equation*}
    \varphi(\rho):=\inf _{u \in \Phi^{-1}(]-\infty, \rho[)} \frac{\left(\sup _{v \in \Phi^{-1}(]-\infty, \rho[)} \Psi(v)\right)-\Psi(u)}{\rho-\Phi(u)} 
\end{equation*}
for every $\rho>\inf _{X} \Phi$.

Then, for every $\rho>\inf _{X} \Phi$ and $\lambda \in( 0,1 / \varphi(\rho))$, the restriction of the functional $J_{\lambda}:=\Phi-\lambda \Psi$ to $\Phi^{-1}(-\infty, \rho)$ admits a global minimum, which is a critical point (local minimum) of $J_{\lambda}$ in $X$.
\end{thm}

We also recall the following classical mountain pass lemma \cite[Theorem 6.1]{Struwe}. 

\begin{thm}\label{MP thm} 
Let $(X,\|\cdot\|)$ be a Banach space and let $F : X\to \mathbb{R}$ be a $C^{1}$ functional. Suppose that $F$ satisfies the Palais-Smale condition and the following assumptions hold:

(i) $F(0)=0$;

(ii) There exist $\rho,$ $\alpha>0$ such that $F(u) \geq \alpha$ for all $u\in X$ with $\|u\|=\rho$;

(iii) There exists a function $v \in X$ such that $\left\|v\right\| \geq \rho$ and $F\left(v\right)<\alpha$.

Define the set of paths joining $u = 0$ with $u = v$:
\begin{equation*}
    \Gamma:=\left\{\gamma \in C([0,1] ; X): \gamma(0)=0, \gamma(1)=v\right\}.
\end{equation*}
Then,
\begin{equation*}
    c:=\inf _{\gamma \in \Gamma} \sup _{u \in \gamma} F(u) \geq \alpha
\end{equation*}
is a critical value of $F$.
\end{thm}


\section{Concentration-compactness principle for the mixed local-nonlocal operators} \label{sec3}

In this section, we establish the concentration-compactness principle for the mixed local-nonlocal operators ( that is, Theorem \ref{CCP thm}). The proof follows the arguments of \cite[Theorem 1.1]{F-P} and relies on the detailed analysis of the exact behavior of weakly convergent sequences of $X(\Omega)$ in the space of measures.  

\begin{proof}[Proof of Theorem \ref{CCP thm}]
    Let $\{u_{k}\}$ be a weakly convergent sequence in $X(\Omega)$ with weak limit $u\in X(\Omega)$. Then by the continuous embedding \cite{MGS2025ArX}, $X(\Omega) \hookrightarrow L^{p}\left(\Omega,|x|^{-p\theta }\right)$, the sequence $\{u_{k}\}$ converges weakly to $u$ also in $L^{p}\left(\Omega,|x|^{-p\theta}\right)$. Moreover, the measures
     \begin{equation*}
\left\{\left|\nabla u_{k}(x)\right|^{p} d x +\left(\int_{\mathbb{R}^N}\frac{|u_k(x)-u_k(y)|^p}{|x-y|^{N+ps}}dy\right)dx\right\} \quad \text { and } \quad \left\{\frac{\left|u_{k}(x)\right|^{p}}{|x|^{p\theta}} dx\right\}
\end{equation*}
are uniformly tight in $k$. Indeed, since $\Omega$ is bounded, one can choose a bounded open set $\mathcal{O} \subset \mathbb{R}^{N}$ satisfying $\overline{\Omega} \subset \mathcal{O}$. Then, we deduce that $u_{k}(x)=0$ for almost every $x \in \mathbb{R}^{N} \backslash \mathcal{O}$. Using this, we obtain 
\begin{align*}
\int_{\mathbb{R}^{N} \backslash \mathcal{O}}\left|\nabla u_{k}(x)\right|^{p} d x&+\int_{\mathbb{R}^{N} \backslash \mathcal{O}}\int_{\mathbb{R}^N}\frac{|u_k(x)-u_k(y)|^p}{|x-y|^{N+ps}}dy d x\\& =\int_{\mathbb{R}^{N} \backslash \mathcal{O}} \left(\int_{\mathbb{R}^{N}} \frac{\left|u_{k}(x)-u_{k}(y)\right|^{p}}{|x-y|^{N+p s}} d y\right) d x\\
& =\int_{\mathbb{R}^{N} \backslash \mathcal{O}} \left(\int_{\mathbb{R}^{N}} \frac{\left|u_{k}(y)\right|^{p}}{|x-y|^{N+p s}} d y\right)d x\\
& =\int_{\mathbb{R}^{N} \backslash \mathcal{O}} \left(\int_{\Omega} \frac{\left|u_{k}(y)\right|^{p}}{|x-y|^{N+p s}} d y\right) d x\\
& \leq \int_{\mathbb{R}^{N} \backslash \mathcal{O}} \frac{d x}{\operatorname{dist}(x, \overline{\Omega})^{N+p s}} \int_{\Omega}\left|u_{k}(y)\right|^{p} d y \\
& \leq\left\|u_{k}\right\|_{p}^{p}\left(\int_{\mathbb{R}^{N} \backslash \mathcal{O}} \frac{d x}{\operatorname{dist}(x, \overline{\Omega})^{N+p s}}\right) \\
& \leq \sup _{k}\left\|u_{k}\right\|_{p}^{p}\left(\int_{\mathbb{R}^{N} \backslash \mathcal{O}} \frac{d x}{\operatorname{dist}(x, \overline{\Omega})^{N+p s}}\right) \leq C
\end{align*}
for some constant $C>0$, since $\operatorname{dist}\left(\mathbb{R}^{N} \backslash \mathcal{O}, \overline{\Omega}\right)>0$ and $N+p s>N$.

Arguing as above, we also obtain the tightness of the measure $\{|u_{k}|^p /|x|^{p\theta }\}$. Therefore, utilizing  \cite[Theorem 8.6.2]{Bogachev}, we conclude that there exist two finite positive measures $\omega$ and $\sigma$ in $\mathbb{R}^{N}$ such that \eqref{ccp pr 1} holds.

Now we set $v_{k}=u_{k}-u$. From the weak convergence, we deduce that $v_{k} \rightharpoonup 0$ in $X(\Omega)$ as $k \rightarrow \infty$. Using the same argument as above, we derive the existence of two positive measures $\widehat{\omega}$ and $\widehat{\sigma}$ on $\mathbb{R}^{N}$ such that
\begin{align}\label{ccp pr 4}
\begin{split}
    \left|\nabla v_{k}(x)\right|^{p} d x +\left(\int_{\mathbb{R}^N}\frac{|v_k(x)-v_k(y)|^p}{|x-y|^{N+ps}}dy\right)dx \stackrel{*}{\rightharpoonup} \widehat{\omega} \\
     \text{ and } \quad \frac{\left|v_{k}(x)\right|^{p} }{|x|^{p\theta}}dx \stackrel{*}{\rightharpoonup} \widehat{\sigma} \quad \text { in } \mathcal{M}\left(\mathbb{R}^{N}\right).
\end{split}
\end{align}
Furthermore, by the compact embedding \eqref{Rel-Kon comp emb}, the sequence $\{u_{k}\}$ strongly converges to $u$ in $L^{p}(\Omega)$. By extending the functions trivially to the whole $\mathbb{R}^{N}$, this strong convergence also holds in $L^{p}\left(\mathbb{R}^{N}\right)$. Then \cite[ Theorem 4.9]{Brezis} implies that there exists $h \in L^{p}(\Omega)$ such that, up to a subsequence, still relabeled $\{u_{k}\}$, we have
\begin{equation}\label{ccp pr 4.1}
u_{k} \rightarrow u \text { a.e. in } \Omega, \quad\left|u_{k}\right| \leq h \text { a.e. in } \Omega \text { and all } k . 
\end{equation}
Thus, applying the Brezis-Lieb lemma (see \cite{B-L}), we obtain
\begin{align*}
\int_{\Omega}|\varphi(x)|^{p} d \sigma-\int_{\Omega}\frac{|\varphi u|^p}{|x|^{p\theta}}dx & =\lim _{k \rightarrow \infty}\int_{\Omega}\frac{|\varphi u_k|^p}{|x|^{p\theta}}dx-\int_{\Omega}\frac{|\varphi u|^p}{|x|^{p\theta}}dx \\
& =\lim _{k \rightarrow \infty}\int_{\Omega}\frac{|\varphi v_k|^p}{|x|^{p\theta}}dx=\int_{\Omega}|\varphi(x)|^{p} d \widehat{\sigma}
\end{align*}
for any $\varphi \in C_{0}^{\infty}(\Omega)$. Since $\varphi \in C_{0}^{\infty}(\Omega)$ is arbitrary, we conclude that
\begin{equation}\label{to pr 1.5}
    \sigma=\widehat{\sigma}+ \frac{|u(x)|^{p}}{|x|^{p\theta}}dx.
\end{equation}

Now we are in a position to prove \eqref{ccp pr 2}. For this, it is enough to show that $\widehat{\sigma}=\sigma_0\delta_0$ in \eqref{to pr 1.5}.

We fix $\varphi \in C_{0}^{\infty}(\Omega)$ and $\varepsilon>0$. Then there exists a positive constant $C_{\varepsilon}>0$ such that 
\begin{equation*}
    |a+b|^{p} \leq(1+\varepsilon)|a|^{p}+C_{\varepsilon}|b|^{p}
\end{equation*}
for all $a, b \in \mathbb{R}$. Thus, applying the Leibniz formula, putting $a=(v_k(x)-v_k(y))\varphi(x)$ and $b=v_k(y)(\varphi(x)-\varphi(y))$ in the above inequality, we obtain for all $k$ that
\begin{align}\label{ccp pr 5}
\begin{split}
\left[ v_{k}\varphi\right]_{s, p}^{p}&=\int_{\mathbb{R}^{N}}\left(\int_{\mathbb{R}^N}\frac{|(v_k\varphi)(x)-(v_k\varphi)(y)|^p}{|x-y|^{N+ps}}dy\right) d x \\
&\leq(1+\varepsilon) \int_{\mathbb{R}^{N}}\left(\int_{\mathbb{R}^N}\frac{|v_k(x)-v_k(y)|^p}{|x-y|^{N+ps}}dy\right)|\varphi(x)|^p d x\\
    &+C_{\varepsilon} \int_{\mathbb{R}^{N}}\left(\int_{\mathbb{R}^N}\frac{|\varphi(x)-\varphi(y)|^p}{|x-y|^{N+ps}}dy\right)|v_k(x)|^p d x.
\end{split}
\end{align}
Moreover, we have
\begin{equation}\label{ccp pr 6}
\int_{\mathbb{R}^{N}} \frac{|\varphi(x)-\varphi(y)|^{p}}{|x-y|^{N+p s}} d y \leq 2^{p}\|\varphi\|_{C^{1}\left(\mathbb{R}^{N}\right)}^{p} \int_{\mathbb{R}^{N}} \frac{\min \left\{1,|x-y|^{p}\right\}}{|x-y|^{N+p s}} d y \leq C_{\varphi,s},
\end{equation}
for some $C_{\varphi,s}>0$ depending on $N, p$ and $s$. 
Thus, combining \eqref{ccp pr 5} and \eqref{ccp pr 6}, we obtain 
\begin{equation}\label{ccp pr 7}
\left[ v_{k}\varphi\right]_{s, p}^{p} \leq(1+\varepsilon) \int_{\mathbb{R}^{N}}\left(\int_{\mathbb{R}^N}\frac{|v_k(x)-v_k(y)|^p}{|x-y|^{N+ps}}dy\right)|\varphi(x)|^{p} d x+C_{\varepsilon, \varphi, s}\left\|v_{k}\right\|_{p}^{p} 
\end{equation}
for some constant $C_{\varepsilon, \varphi,s}>0$.

Arguing as above, we also obtain 
\begin{equation}\label{ccp pr 8}
    \left\|\nabla\left(v_k \varphi\right)\right\|_p^p \leq(1+\varepsilon) \int_{\Omega}\left|\nabla v_k\right|^p|\varphi|^p dx+C_{\varepsilon, \varphi}\left\|v_k\right\|_p^p
\end{equation}
for some $C_{\varepsilon, \varphi}>0$.

Combining \eqref{ccp pr 7} and \eqref{ccp pr 8} gives
\begin{align*}
\begin{split}
    \left\|\nabla\left(v_k \varphi\right)\right\|_p^p+\left[ v_{k}\varphi\right]_{s, p}^{p} &\leq(1+\varepsilon) \Bigg(\int_{\Omega}\left|\nabla v_k(x)\right|^p|\varphi(x)|^p dx\\
&+\int_{\mathbb{R}^{N}}\left(\int_{\mathbb{R}^N}\frac{|v_k(x)-v_k(y)|^p}{|x-y|^{N+ps}}dy\right)|\varphi(x)|^{p} d x\Bigg)+2\mathrm{max}(C_{\varepsilon, \varphi},C_{\varepsilon, \varphi,s})\left\|v_k\right\|_p^p. 
\end{split}
\end{align*}
Therefore, applying the interpolated Hardy inequality \eqref{mixed Hardy use} to the sequence $\{v_{k}\varphi\}\subset X(\Omega)$, we get
\begin{align}\label{ccp pr 9}
\begin{split}
    \bar{\mu}(\theta)&\int_{\Omega}\frac{|v_{k}\varphi|^p}{|x|^{p\theta}}dx \leq\left\|\nabla\left(v_k \varphi\right)\right\|_p^p+\left[ v_{k}\varphi\right]_{s, p}^{p} \leq(1+\varepsilon) \Bigg(\int_{\Omega}\left|\nabla v_k(x)\right|^p|\varphi(x)|^p dx\\
&+\int_{\mathbb{R}^{N}}\left(\int_{\mathbb{R}^N}\frac{|v_k(x)-v_k(y)|^p}{|x-y|^{N+ps}}dy\right)|\varphi(x)|^{p} d x\Bigg)+2\mathrm{max}(C_{\varepsilon, \varphi},C_{\varepsilon, \varphi,s})\left\|v_k\right\|_p^p.
\end{split}
\end{align}
Finally, passing to the limit on both sides of the above inequality and using \eqref{ccp pr 4}, together with the fact that $v_{k}\rightarrow 0$ in $L^{p}(\Omega)$ as $k \rightarrow \infty$, we obtain 
\begin{equation*}
    \int_{\Omega}|\varphi(x)|^{p} d \widehat{\sigma} \leq \frac{1+\varepsilon}{\bar{\mu}(\theta)} \int_{\Omega}|\varphi(x)|^{p} d \widehat{\omega},
\end{equation*}
which implies that the measure $\widehat{\sigma}$ is absolutely continuous with respect to $\widehat{\omega}$. Therefore, it follows from \cite[Lemma 1.2]{Lions-1985} that the measure $\widehat{\sigma}$ decomposes into a sum of Dirac masses. 

Now we show that the measure $\widehat{\sigma}$ is concentrated at 0. For this, we assume that $0 \notin \operatorname{supp}(\varphi)$. Then we have $|\varphi(x)|^{p} /|x|^{p\theta}\in L^{\infty}(\operatorname{Supp}(\varphi))$. Using this and the compact embedding \eqref{Rel-Kon comp emb}, we obtain that
\begin{equation}\label{ccp pr 444}
    \int_{\Omega}\frac{|\varphi v_{k}|^p}{|x|^{p\theta}}dx=\int_{\operatorname{Supp}(\varphi)} \frac{|\varphi(x)|^{p}}{|x|^{p\theta}}\left|v_{k}(x)\right|^{p} d x \leq C \int_{\operatorname{Supp}(\varphi)}\left|v_{k}(x)\right|^{p} d x \rightarrow 0 \quad \text{ as }k \rightarrow \infty .
\end{equation}
Combining \eqref{ccp pr 444} and \eqref{ccp pr 4}, we deduce that $\int_{\Omega}|\varphi(x)|^{p} d \widehat{\sigma}=0$, which means that $\widehat{\sigma}$ is a measure concentrated at $0$. Hence, we have $\widehat{\sigma}=\sigma_{0} \delta_{0}$ completing the proof of \eqref{ccp pr 2}.

Next, we prove \eqref{ccp pr 3}. Arguing as in \eqref{ccp pr 9}, replacing $v_{k}$ by $u_{k}$, and using \eqref{ccp pr 4} and \eqref{ccp pr 4.1}, we get
\begin{align}\label{ccp pr 13}
\begin{split}
\bar{\mu}(\theta)&\left(\int_{\Omega}|\varphi(x)|^{p} d \sigma\right) \leq(1+\varepsilon) \int_{\Omega}|\varphi(x)|^{p} d \omega\\
&+C_{\varepsilon} \left(\int_{\Omega}\left|\nabla \varphi(x)\right|^{p}|u(x)|^{p} d x+\int_{\mathbb{R}^{N}}\left(\int_{\mathbb{R}^N}\frac{|\varphi(x)-\varphi(y)|^p}{|x-y|^{N+ps}}dy\right)|u(x)|^{p} d x\right) 
\end{split}
\end{align}
as $k \rightarrow \infty$.

Let us now consider the test function $\varphi_{\tilde{\varepsilon}}(x)=\varphi(x / \tilde{\varepsilon})$  for $\tilde{\varepsilon}>0$ sufficiently small, where $\varphi \in C_{0}^{\infty}\left(\mathbb{R}^{N}\right)$, with $0 \leq \varphi \leq 1$, $\varphi(0)=1,$ and  $\operatorname{supp}(\varphi)=B(0,1)$. Since $\sigma \geq \sigma_{0} \delta_{0}$, choosing $\varphi_{\tilde{\varepsilon}}$ as a test function in \eqref{ccp pr 13}, we obtain
\begin{align}\label{ccp pr 14}
\begin{split}
    0 &\leq \bar{\mu}(\theta) \sigma_{0} \leq(1+\varepsilon) \omega(B(0, \tilde{\varepsilon}))\\
&+C_{\varepsilon} \left(\int_{\Omega}\left|\nabla \varphi_{\tilde{\varepsilon}}(x)\right|^{p}|u(x)|^{p} d x+\int_{\mathbb{R}^{N}}\left(\int_{\mathbb{R}^N}\frac{|\varphi_{\tilde{\varepsilon}}(x)-\varphi_{\tilde{\varepsilon}}(y)|^p}{|x-y|^{N+ps}}dy\right)|u(x)|^{p} d x\right).
\end{split}
\end{align}
As already proved in \cite[page 436]{F-P}, we have 
\begin{equation*}
    \lim _{\tilde{\varepsilon} \rightarrow 0} \int_{\mathbb{R}^{N}}\left(\int_{\mathbb{R}^N}\frac{|\varphi_{\tilde{\varepsilon}}(x)-\varphi_{\tilde{\varepsilon}}(y)|^p}{|x-y|^{N+ps}}dy\right)|u(x)|^{p} d x=0.
\end{equation*}
Moreover, applying H\"older's inequality yields
\begin{equation*}
    \lim _{\tilde{\varepsilon} \rightarrow 0} \int_{\Omega}\left|\nabla \varphi_{\tilde{\varepsilon}}(x)\right|^{p}|u(x)|^{p} d x=0.
\end{equation*}
Hence, letting $\tilde{\varepsilon} \rightarrow 0$ and $\varepsilon \rightarrow 0$ in \eqref{ccp pr 14}, we obtain $0 \leq \bar{\mu}(\theta) \sigma_{0} \leq \omega_{0}$. Now, using the weak lower semicontinuity of the norm given in \eqref{global norm}, we derive 
\begin{align*}
\liminf_{k\to\infty}&\left(\int_{\Omega}\left|\nabla u_k(x)\right|^{p} d x +\iint_{\mathbb{R}^{2N}}\frac{|u_k(x)-u_k(y)|^p}{|x-y|^{N+ps}}dxdy\right)\\
    &\geq\int_{\Omega}\left|\nabla u(x)\right|^{p} d x +\iint_{\mathbb{R}^{2N}}\frac{|u(x)-u(y)|^p}{|x-y|^{N+ps}}dxdy.
\end{align*}  
Combining this with \eqref{ccp pr 1}, gives
\begin{align} \label{eqshe}
    \omega \geq\left|\nabla u(x)\right|^{p} d x +\left(\int_{\mathbb{R}^N}\frac{|u(x)-u(y)|^p}{|x-y|^{N+ps}}dy\right)dx.
\end{align}
Using \eqref{eqshe} and the orthogonality of the measures $\left|\nabla u(x)\right|^{p} d x +\left(\int_{\mathbb{R}^N}\frac{|u(x)-u(y)|^p}{|x-y|^{N+ps}}dy\right)dx$ and $\omega_{0} \delta_{0}$, we conclude that \eqref{ccp pr 3} holds.
\end{proof}

We now define the functional $\mathcal{I}_{\mu}: X(\Omega) \rightarrow \mathbb{R}$ as follows
\begin{equation}\label{Hardy functional}
    \mathcal{I}_{\mu}(u):=\frac{1}{p}\left(\int_{\Omega}\left|\nabla u\right|^p dx+\iint_{\mathbb{R}^{2N}} \frac{|u(x)-u(y)|^{p}}{|x-y|^{N+sp}} dxdy-\mu \int_{\Omega} \frac{\left|u\right|^p}{|x|^{p\theta}} dx\right).
\end{equation}
Using Theorem \ref{CCP thm}, we establish that $\mathcal{I}_{\mu}$ is a coercive and weakly lower semicontinuous functional on $X(\Omega)$.

\begin{lem}\label{wls lem}
If $\mu<\bar{\mu}(\theta)$, then the functional $\mathcal{I}_{\mu}$ is coercive and weakly lower semicontinuous on $X(\Omega)$.
\end{lem}
\begin{proof}
Let $\{u_k\}$ be a sequence such that $u_k \rightharpoonup u$ in $X(\Omega)$.   By Theorem \ref{CCP thm} there exist two positive measures $\omega$ and $\sigma$, verifying \eqref{ccp pr 1}-\eqref{ccp pr 3}. Since $\mu < \bar{\mu}(\theta)$, Theorem \ref{CCP thm} yields
\begin{align*}
\liminf _{k \rightarrow \infty} \mathcal{I}_{\mu}\left(u_k\right) & =\liminf_{k \rightarrow \infty}\frac{1}{p} \left(\int_{\Omega}\left|\nabla u_k\right|^p dx+\iint_{\mathbb{R}^{2N}} \frac{|u_k(x)-u_k(y)|^{p}}{|x-y|^{N+sp}} dxdy-\mu \int_{\Omega} \frac{\left|u_k\right|^p}{|x|^{p\theta}} dx\right) \\
& \geq \frac{1}{p}\left(\int_{\Omega}\left|\nabla u\right|^p dx+\iint_{\mathbb{R}^{2N}} \frac{|u(x)-u(y)|^{p}}{|x-y|^{N+sp}} dxdy+\omega_0-\mu\left( \int_{\Omega} \frac{\left|u\right|^p}{|x|^{p\theta}} dx+\sigma_0\right) \right)\\
& = \mathcal{I}_{\mu}(u)+\omega_0-\mu \sigma_0 \\
& \geq \mathcal{I}_{\mu}(u)+\left(\bar{\mu}(\theta)-\mu\right) \sigma_0 \\
& \geq \mathcal{I}_{\mu}(u).
\end{align*}
This implies that the functional $\mathcal{I}_{\mu}$ is weakly lower semicontinuous. Now, using the mixed interpolated Hardy inequality \eqref{mixed Hardy use}, we obtain 
\begin{equation*}
\mathcal{I}_{\mu}(u) \geq \frac{1}{p}\left(1-\frac{\mu}{\bar{\mu}(\theta)}\right)\|u\|_{X}^{p}, ~\text{ for all }~u \in X(\Omega).
\end{equation*}
Hence, we have $\mathcal{I}_{\mu}(u) \rightarrow \infty$ as $\|u\|_X \rightarrow \infty$, which implies that $\mathcal{I}_\mu$ is coercive.
\end{proof}

\section{Proof of Theorem \ref{main result thm}} \label{sec4}

This section is devoted to the proof of Theorem \ref{main result thm}. We begin by stating the definition of a weak solution to problem \eqref{main problem 0}.

\begin{definition}
     We say that a function $u: \Omega \rightarrow \mathbb{R}$ is a weak solution of \eqref{main problem 0} if $u \in X(\Omega)$ satisfies
\begin{align*}
& \int_{\Omega}|\nabla u(x)|^{p-2} \nabla u(x) \nabla \phi(x) d x+\iint_{\mathbb{R}^{2N}} \frac{\mathcal{A}(u(x)-u(y))(\phi(x)-\phi(y))}{|x-y|^{N+sp}} dxdy\\
&-\mu \int_{\Omega} \frac{|u(x)|^{p-2}u(x)}{|x|^{p}} \phi(x) d x-\lambda \int_{\Omega} f(x, u(x)) \phi(x) d x=0,~\forall\,\phi \in X(\Omega).
\end{align*}
\end{definition}
Here, for $1<p<\infty$, the function $\mathcal{A}: \mathbb{R} \rightarrow \mathbb{R}$ is defined by
\begin{equation*}
    \mathcal{A}(t)= \begin{cases}|t|^{p-2} t & \text { if } t \neq 0, \\ 0 & \text { if } t=0.\end{cases}
\end{equation*}

We define the energy functional $\mathcal{J}_{\mu, \lambda}: X(\Omega) \rightarrow \mathbb{R}$ associated with \eqref{main problem 0}, as
\begin{align}\label{def j}
\mathcal{J}_{\mu, \lambda}(u)&=\frac{1}{p}\left(\int_{\Omega}\left|\nabla u\right|^p dx+\iint_{\mathbb{R}^{2N}} \frac{|u(x)-u(y)|^{p}}{|x-y|^{N+sp}} dxdy\right) -\frac{\mu}{p} \int_{\Omega} \frac{\left|u\right|^p}{|x|^{p\theta}} dx\nonumber\\
&\quad\quad\quad\quad\quad-\lambda \int_{\Omega} F(x, u(x)) d x,
\end{align}
where $F(x, \xi)=\int_{0}^{\xi} f(x, t) d t$, for every $(x, \xi) \in \Omega \times \mathbb{R}$. 

In order to apply Theorem \ref{abstr CPT}, we decompose the functional $\mathcal{J}_{\mu, \lambda}$ as
\begin{equation*}
\mathcal{J}_{\mu, \lambda}(u)=\mathcal{I}_{\mu}(u)-\lambda \mathcal{P}(u),
\end{equation*}
where $\mathcal{I}_{\mu}$ is defined in \eqref{Hardy functional} and
\begin{equation*}
    \mathcal{P}(u):=\int_{\Omega} F(x, u(x)) d x.
\end{equation*}

Observe that $\mathcal{I}_{\mu}$ is a G\^ateaux differentiable functional in $X(\Omega)$ with derivative given by
\begin{align*}
\langle \mathcal{I}_{\mu}^{\prime}(u),\phi \rangle&=\int_{\Omega}|\nabla u(x)|^{p-2} \nabla u(x)  \nabla \phi(x) d x\\
&+\iint_{\mathbb{R}^{2N}} \frac{\mathcal{A}(u(x)-u(y))(\phi(x)-\phi(y))}{|x-y|^{N+sp}} dxdy\\
&-\mu \int_{\Omega} \frac{|u(x)|^{p-2}u(x)}{|x|^{p}}  \phi(x) d x
\end{align*}
for every $\phi \in  X(\Omega)$. According to Lemma \ref{wls lem}, $\mathcal{I}_{\mu}$ is weakly lower semicontinuous and coercive for $\mu \in (0, \bar{\mu}(\theta))$, and clearly $\inf _{u \in X(\Omega)} \mathcal{I}_\mu(u)=0$. Indeed, using the interpolated Hardy inequality \eqref{mixed Hardy use}, we have for every $u \in  X(\Omega)$ that
\begin{equation}\label{bound for Imu}
\left(\frac{\bar{\mu}(\theta)-\mu}{p \bar{\mu}(\theta)}\right)\|u\|_{X}^{p} \leq  \mathcal{I}_{\mu}(u) \leq \frac{\|u\|_{X}^{p}}{p}.
\end{equation}

The functional $\mathcal{P}$ is well defined and continuously G\^ateaux differentiable, with G\^ateaux derivative given by
\begin{equation*}
\langle \mathcal{P}^{\prime}(u), \phi\rangle=\int_{\Omega} f(x, u(x)) \phi(x) d x
\end{equation*}
for every $\phi \in  X(\Omega)$. Moreover, using the fact the $f$ is Carath\'eodory and $X(\Omega)$ is compactly embedded in $L^r(\Omega),~1\leq r<p^*$, we conclude that $\mathcal{P}':X(\Omega) \rightarrow X(\Omega)^{*}$ is a compact operator, where $X(\Omega)^{*}$ is the dual space of $X(\Omega)$. Therefore, for any $u, \phi \in X(\Omega)$, we have 
\begin{align*}
\langle \mathcal{J}_{\mu,\lambda}^{\prime}(u), \phi\rangle
&=\int_{\Omega}|\nabla u(x)|^{p-2} \nabla u(x)  \nabla \phi(x) d x\\
&+\iint_{\mathbb{R}^{2N}} \frac{\mathcal{A}(u(x)-u(y))(\phi(x)-\phi(y))}{|x-y|^{N+sp}} dxdy\\
&-\mu \int_{\Omega} \frac{|u(x)|^{p-2}u(x)}{|x|^{p}}  \phi(x) d x-\lambda \int_{\Omega} f(x, u(x)) \phi(x) d x,
\end{align*}
and the critical points of $\mathcal{J}_{\mu, \lambda}$ are exactly the weak solutions of problem \eqref{main problem 0}.

\begin{proof}[Proof of Theorem \ref{main result thm}]
We split the proof into three steps.\\

\noindent\textbf{Step 1.} Problem \eqref{main problem 0} admits at least one nontrivial weak solution $u_{\lambda} \in X(\Omega)$.\\

We fix $\mu \in (0, \bar{\mu}(\theta))$ and $\lambda \in (0, \Lambda),$ where $\Lambda$ is defined by \eqref{value of lambdamu}. Since $0<\lambda<\Lambda$, there exists $\bar{\sigma}>0$ such that
\begin{equation}\label{PRM bound for lam}
\lambda<\Lambda(\bar{\sigma}):=\frac{q \bar{\sigma}^{p-1}}{q a_{1} C_{1}\left(\frac{p \bar{\mu}(\theta)}{\bar{\mu}(\theta)-\mu}\right)^{1 / p}+a_{2} C_{q}^{q}\left(\frac{p \bar{\mu}(\theta)}{\bar{\mu}(\theta)-\mu}\right)^{q / p} \bar{\sigma}^{q-1}}. 
\end{equation}

On the other hand, from \eqref{cond on f}, it follows that
\begin{equation}\label{bound for F}
F(x, \xi) \leq a_{1}|\xi|+a_{2} \frac{|\xi|^{q}}{q} 
\end{equation}
for every $(x, \xi) \in \Omega \times \mathbb{R}$.
Consequently, using \eqref{bound for F} yields
\begin{equation}\label{functional P bound 1}
\mathcal{P}(u)=\int_{\Omega} F(x, u(x)) d x \leqslant a_{1}\|u\|_{L^{1}(\Omega)}+\frac{a_{2}}{q}\|u\|_{L^{q}(\Omega)}^{q}.
\end{equation}
Moreover, from \eqref{bound for Imu}, we have
\begin{equation}\label{PRM bound for u}
\|u\|_{X}<\left(\frac{p \bar{\mu}(\theta) \rho}{\bar{\mu}(\theta)-\mu}\right)^{1 / p} 
\end{equation}
for every $u \in  X(\Omega)$ such that $\mathcal{I}_{\mu}(u)<\rho$ with $\rho \in (0,+\infty)$. Then, using \eqref{Sobolev emb} and \eqref{PRM bound for u}, we obtain from \eqref{functional P bound 1} that
\begin{equation*}
\mathcal{P}(u)<a_{1}C_{1} \left(\frac{p \bar{\mu}(\theta)}{\bar{\mu}(\theta)-\mu}\right)^{1 / p} \rho^{1 / p}+a_{2} \frac{C_{q}^{q}}{q}\left(\frac{p \bar{\mu}(\theta)}{\bar{\mu}(\theta)-\mu}\right)^{q / p} \rho^{q / p},
\end{equation*}
for every $u \in  X(\Omega)$ such that $\mathcal{I}_\mu(u)<\rho$. Thus, we get
\begin{equation}\label{sup bound of P}
\sup _{u \in \mathcal{I}_{\mu}^{-1}((-\infty, \rho))} \mathcal{P}(u) \leq  a_{1}C_{1} \left(\frac{p \bar{\mu}(\theta)}{\bar{\mu}(\theta)-\mu}\right)^{1 / p} \rho^{1 / p}+a_{2} \frac{C_{q}^{q}}{q}\left(\frac{p \bar{\mu}(\theta)}{\bar{\mu}(\theta)-\mu}\right)^{q / p} \rho^{q / p}.
\end{equation}

Now, for $\rho \in (0,+\infty)$, we consider the function
\begin{equation*}
\psi(\rho):=\frac{\sup _{u \in \mathcal{I}_{\mu}^{-1}((-\infty, \rho))} \mathcal{P}(u)}{\rho}.
\end{equation*}
Then, from \eqref{sup bound of P}, we obtain
\begin{equation*}
\psi(\rho) \leq a_{1} C_{1} \left(\frac{p \bar{\mu}(\theta)}{\bar{\mu}(\theta)-\mu}\right)^{1 / p} \rho^{1 / p-1}+a_{2} \frac{C_{q}^{q}}{q}\left(\frac{p \bar{\mu}(\theta)}{\bar{\mu}(\theta)-\mu}\right)^{q / p} \rho^{q / p-1} 
\end{equation*}
for every $\rho>0$.
In particular, taking $\rho:=\bar{\sigma}^{p}$, we get
\begin{equation}\label{PMR bound for kappa}
\psi\left(\bar{\sigma}^{p}\right) \leq a_{1} C_{1} \left(\frac{p \bar{\mu}(\theta)}{\bar{\mu}(\theta)-\mu}\right)^{1 / p} \bar{\sigma}^{1-p}+a_{2} \frac{C_{q}^{q}}{q}\left(\frac{p \bar{\mu}(\theta)}{\bar{\mu}(\theta)-\mu}\right)^{q / p} \bar{\sigma}^{q-p} .
\end{equation}

Now, we define the function
\begin{equation*}
\varphi\left(\bar{\sigma}^{p}\right):=\inf _{u \in \mathcal{I}_{\mu}^{-1}((-\infty, \bar{\sigma}^{p}))} \frac{\left(\sup _{v \in \mathcal{I}_{\mu}^{-1}((-\infty, \bar{\sigma}^{p}))} \mathcal{P}(v)\right)-\mathcal{P}(u)}{\bar{\sigma}^{p}-\mathcal{I}_{\mu}(u)} .
\end{equation*}
Since $u_{0} \in \mathcal{I}_{\mu}^{-1}((-\infty, \bar{\sigma}^{p}))$ and $\mathcal{I}_{\mu}\left(u_{0}\right)=\mathcal{P}\left(u_{0}\right)=0$, where $u_{0} \in X(\Omega)$ is the identically zero function, we conclude that 
\begin{equation}\label{ineq for phi and chi}
\varphi\left(\bar{\sigma}^{p}\right) \leq \psi\left(\bar{\sigma}^{p}\right),
\end{equation}
where
\begin{equation*}
\psi(\bar{\sigma}^{p}):=\frac{\sup _{v \in \mathcal{I}_{\mu}^{-1}((-\infty, \bar{\sigma}^{p}))} \mathcal{P}(v)}{\bar{\sigma}^{p}}.
\end{equation*}

Finally, using \eqref{ineq for phi and chi}, \eqref{PMR bound for kappa}, and \eqref{PRM bound for lam}, we derive
\begin{equation*}
\varphi\left(\bar{\sigma}^{p}\right) \leq \psi\left(\bar{\sigma}^{p}\right) \leq a_{1} C_{1} \left(\frac{p \bar{\mu}(\theta)}{\bar{\mu}(\theta)-\mu}\right)^{1 / p} \bar{\sigma}^{1-p}+a_{2} \frac{C_{q}^{q}}{q}\left(\frac{p \bar{\mu}(\theta)}{\bar{\mu}(\theta)-\mu}\right)^{q / p} \bar{\sigma}^{q-p}<\frac{1}{\lambda} ,
\end{equation*}
which implies
\begin{equation*}
\lambda \in\left( 0, \frac{q \bar{\sigma}^{p-1}}{q a_{1} C_{1}\left(\frac{p \bar{\mu}(\theta)}{\bar{\mu}(\theta)-\mu}\right)^{1 / p}+a_{2} C_{q}^{q}\left(\frac{p \bar{\mu}(\theta)}{\bar{\mu}(\theta)-\mu}\right)^{q / p} \bar{\sigma}^{q-1}}\right) \subseteq\left( 0,\frac{1}{\varphi\left(\bar{\sigma}^{p}\right)}\right).
\end{equation*}

Now taking $\Phi:=\mathcal{I}_\mu$, $\Psi:=\mathcal{P}$, and $\rho:=\bar{\sigma}^{p}$, we see that all the assumptions of Theorem \ref{abstr CPT} are satisfied. Hence, applying Theorem \ref{abstr CPT}, we conclude that there exists a function $u_{\lambda} \in \mathcal{I}_{\mu}^{-1}((-\infty, \bar{\sigma}^{p}))$ such that
\begin{equation*}
\mathcal{J}_{\mu, \lambda}^{\prime}\left(u_{\lambda}\right)=\mathcal{I}_{\mu}^{\prime}\left(u_{\lambda}\right)-\lambda \mathcal{P}^{\prime}\left(u_{\lambda}\right)=0.
\end{equation*}
In particular, $u_{\lambda}$ is a global minimum of the restriction of $\mathcal{J}_{\mu, \lambda}$ to $\mathcal{I}_{\mu}^{-1}((-\infty, \bar{\sigma}^{p}))$. Moreover, since $f(x, 0) \neq 0$ in $\Omega$, we have $\mathcal{P}^{\prime}(0)\neq0$. This implies $\mathcal{J}_{\mu, \lambda}^{\prime}(0)\neq 0$, since $\mathcal{I}_{\mu}^{\prime}(0)=0$. Therefore, $0$ is not a critical point of $\mathcal{J}_{\mu, \lambda}$. Thus, we have $u_{\lambda} \neq 0$. Thus, for $\mu \in (0, \bar{\mu}(\theta))$ and for every $\lambda \in( 0, \Lambda)$ the problem \eqref{main problem 0} admits a nontrivial weak solution $u_{\lambda} \in X(\Omega)$.\\

\noindent\textbf{Step 2.} We prove that $\left\|u_{\lambda}\right\|_{X} \rightarrow 0$ as $\lambda \rightarrow 0^{+}$.

As we have earlier noted  that  $\mathcal{I}_{\mu}$ is coercive, therefore $u_{\lambda} \in \mathcal{I}_{\mu}^{-1}((-\infty, \bar{\sigma}^{p}))$ is bounded in $X(\Omega)$, that is,  we have
\begin{equation*}
    \left\|u_{\lambda}\right\|_{X} \leq K, 
\end{equation*}
for some $K>0$ and for every $\lambda \in( 0, \Lambda)$.

Furthermore, by the compactness of the operator $\mathcal{P}^{\prime}$, there exists a constant $C>0$ such that
\begin{equation}\label{PMR bound for mathP}
\left|\left\langle\mathcal{P}^{\prime}\left(u_{\lambda}\right), u_{\lambda}\right\rangle\right| \leq\left\|\mathcal{P}^{\prime}\left(u_{\lambda}\right)\right\|_{X(\Omega)^{*}}\left\|u_{\lambda}\right\|_{X}<C K^{2}
\end{equation}
for every $\lambda \in( 0, \Lambda)$. 
 
Now, since $u_{\lambda}$ is a critical point of $\mathcal{J}_{\mu,\lambda}$ for every $\lambda \in ( 0, \Lambda)$, we have
$$
\left\langle\mathcal{J}_{\mu,\lambda}^{\prime}\left(u_{\lambda}\right), u_{\lambda}\right\rangle=0,
$$
which implies that
\begin{equation}\label{PMR eq for Imu}
p\mathcal{I}_{\mu}\left(u_{\lambda}\right)=\left\langle\mathcal{I}_{\mu}^{\prime}\left(u_{\lambda}\right), u_{\lambda}\right\rangle=\lambda \left\langle\mathcal{P}^{\prime}\left(u_{\lambda}\right), u_{\lambda}\right\rangle
\end{equation}
for every $\lambda \in( 0, \Lambda)$. Thus, from \eqref{PMR bound for mathP} and \eqref{PMR eq for Imu} we obtain that
\begin{equation}\label{PMR lim for Imu}
\lim _{\lambda \rightarrow 0^{+}} \mathcal{I}_{\mu}\left(u_{\lambda}\right)=0 .
\end{equation}

On the other hand, we have by \eqref{bound for Imu} that
\begin{equation}\label{PMR bound for umu}
\left\|u_{\lambda}\right\|_{X}^{p} \leq \frac{p \bar{\mu}(\theta)}{\bar{\mu}(\theta)-\mu} \mathcal{I}_{\mu}\left(u_{\lambda}\right), 
\end{equation}
for every $\lambda \in(0, \Lambda)$. Therefore, from \eqref{PMR lim for Imu} and \eqref{PMR bound for umu} it follows that
\begin{equation*}
\lim _{\lambda \rightarrow 0^{+}}\left\|u_{\lambda}\right\|_{X}=0.
\end{equation*}

\textbf{Step 3.} We show that the function $\lambda \mapsto \mathcal{J}_{\mu, \lambda}\left(u_{\lambda}\right)$ is negative and strictly decreasing in $(0, \Lambda).$

Note that the restriction of the functional $\mathcal{J}_{\mu,\lambda}$ to $\mathcal{I}_{\mu}^{-1}((-\infty, \bar{\sigma}^{p}))$ admits a global minimum, which is a local minimum of $\mathcal{J}_{\mu,\lambda}$ in $X(\Omega)$. It is clear that  $f(x, 0) \neq 0$ in $\Omega$, implies $\mathcal{P}^{\prime}(0)\neq0$. Therefore, there exists $v\in X(\Omega)$ such that $\mathcal{P}^{\prime}(0)v>0$. Consequently, there exists $t_0>0$ such that $\mathcal{P}(tv)>0$ for all $t\in(0,t_0)$. Thus, using first order Taylor's estimate on $\mathcal{P}$ and $p>1$ we obtain $\mathcal{J}_{\mu,\lambda}(tv)=-\lambda t \mathcal{P}^{\prime}(0)v + o(t)$ for sufficiently small $t>0$. In particular, we choose $t$, sufficiently small so that $\mathcal{J}_{\mu,\lambda}(tv)<0$ and $\mathcal{I}_{\mu}(tv)<\bar{\sigma}^{p}.$ Since, $u_{\lambda}$ is a global minimizer of $\mathcal{J}_{\mu,\lambda}$ on $\mathcal{I}_{\mu}^{-1}((-\infty, \bar{\sigma}^{p}))$, we get $\mathcal{J}_{\mu,\lambda}(u_{\lambda})\leq \mathcal{J}_{\mu,\lambda}(tv)<0.$ Therefore, $\mathcal{J}_{\mu,\lambda}(0)=0$ asserts that $u_{\lambda}$ is not a minimizer of $\mathcal{J}_{\mu,\lambda}$, implying that $u_{\lambda}\neq0$. Thus, we conclude that the map $\lambda \mapsto \mathcal{J}_{\mu,\lambda}\left(u_{\lambda}\right)$ is negative in $(0, \Lambda(\bar{\sigma}))$. 

Next, we prove that the map  $\lambda \mapsto \mathcal{J}_{\mu,\lambda}\left(u_{\lambda}\right)$ is strictly decreasing in $(0, \Lambda)$. For this, we write 
\begin{equation*}
\mathcal{J}_{\mu,\lambda}(u)=\lambda\left(\frac{\mathcal{I}_{\mu}(u)}{\lambda}-\mathcal{P}(u)\right)
\end{equation*}
for every $u \in X(\Omega)$. We fix $\lambda_{1},\lambda_{2}\in(0,\Lambda)$ with $\lambda_{1}<\lambda_{2}$ and assume that $u_{\lambda_{1}}, u_{\lambda_{2}} \in X(\Omega)$ are critical points of $\mathcal{J}_{\mu,\lambda}$. Further, we set
\begin{equation*}
\mathcal{E}_{\lambda_{i}}:=\inf _{u \in \mathcal{I}_{\mu}^{-1}((-\infty, \bar{\sigma}^{p}))}\left(\frac{\mathcal{I}_{\mu}(u)}{\lambda_{i}}-\mathcal{P}(u)\right)=\frac{1}{\lambda_{i}} \mathcal{J}_{\mu,\lambda_i}\left(u_{\lambda_{i}}\right), \quad i=1,2.
\end{equation*}
 As claimed before, we have $\mathcal{E}_{\lambda_{i}}<0$ (for $i=1,2$), and $\mathcal{E}_{\lambda_{2}} \leq \mathcal{E}_{\lambda_{1}}$ since $\lambda_{1}<\lambda_{2}$. Thus, we obtain
\begin{equation*}
\mathcal{J}_{\mu, \lambda_{2}}\left(u_{\lambda_{2}}\right)=\lambda_{2} \mathcal{E}_{\lambda_{2}} \leq \lambda_{2} \mathcal{E}_{\lambda_{1}}<\lambda_{1} \mathcal{E}_{\lambda_{1}}=\mathcal{J}_{\mu, \lambda_{1}}\left(u_{\lambda_{1}}\right),
\end{equation*}
which means that the map $\lambda \mapsto \mathcal{J}_{\mu, \lambda}\left(u_{\lambda}\right)$ is decreasing in $( 0, \Lambda)$.

Finally, since $\lambda\in ( 0, \Lambda)$, is arbitrary, the above conclusions are still true in $(0, \Lambda)$. The proof is complete.
\end{proof}

\begin{remark}
    By direct computation, we conclude that the parameter $\Lambda$ in Theorem \ref{main result thm} is defined as follows:
 \begin{equation*}
    \Lambda= \begin{cases}+\infty & \text { if } 1<q<p, \\ \frac{\bar{\mu}(\theta)-\mu}{a_{2} C_{p}^{p} \bar{\mu}(\theta)} & \text { if } q=p, \\ \frac{q \sigma_\mathrm{max}^{p-1}}{q a_{1} C_{1}\left(\frac{p \bar{\mu}(\theta)}{\bar{\mu}(\theta)-\mu}\right)^{1 / p}+a_{2} C_{q}^{q}\left(\frac{p \bar{\mu}(\theta)}{\bar{\mu}(\theta)-\mu}\right)^{q / p} \sigma_\mathrm{max}^{q-1}} & \text { if } p<q<p^{*},\end{cases}
\end{equation*}
where
\begin{equation*}
    \sigma_{\max}:=\left(\frac{\bar{\mu}(\theta)-\mu}{p \bar{\mu}(\theta)}\right)^{\frac{1}{p}}\left[\frac{qa_{1} C_{1}}{a_{2} C_{q}^{q}}\left(\frac{1-p}{p-q}\right)\right]^{\frac{1}{q-1}} .
\end{equation*}
In particular, if $f$ satisfies condition \eqref{cond on f} at infinity, that is, $\limsup_{|t|\to\infty}\frac{|f(x,t)|}{|t|^{q-1}}<\infty$ with $q \in (1, p)$, then Theorem \ref{main result thm} implies that, for any $\lambda > 0$, problem \eqref{main problem 0} admits at least one nontrivial weak solution.
\end{remark}

\section{Proof of Theorem \ref{main result thm 1}} \label{sec5}
This section is devoted to the proof of Theorem \ref{main result thm 1}. We first state the definition of a weak solution to problem \eqref{main problem 0.1}. 

\begin{definition}
     We say that a function $u: \Omega \rightarrow \mathbb{R}$ is a weak solution of \eqref{main problem 0.1} if $u \in X(\Omega)$ and
\begin{align*}
& \int_{\Omega}|\nabla u(x)|^{p-2} \nabla u(x) \nabla \phi(x) d x+\iint_{\mathbb{R}^{2N}} \frac{\mathcal{A}(u(x)-u(y))(\phi(x)-\phi(y))}{|x-y|^{N+sp}} dxdy\\
&-\mu \int_{\Omega} \frac{|u(x)|^{p-2}u(x)}{|x|^{p}} \phi(x) d x-\lambda \int_{\Omega} |u(x)|^{r-2}u(x) \phi(x) d x=0
\end{align*}
for every $\phi \in X(\Omega)$.
\end{definition}

We define the energy functional $\mathcal{F}_{\mu, \lambda}\in C^1( X(\Omega),\mathbb{R})$ associated with \eqref{main problem 0.1}, as
\begin{equation}\label{MP functional}
    \mathcal{F}_{\mu, \lambda}(u)=\frac{1}{p} \int_{\Omega}|\nabla u|^{p} d x+\iint_{\mathbb{R}^{2N}}  \frac{|u(x)-u(y)|^p}{|x-y|^{N+p s}} d x d y-\frac{\mu}{p} \int_{\Omega} \frac{|u|^{p}}{|x|^{p\theta}} d x-\frac{\lambda}{r} \int_{\Omega}|u|^{r} d x.
\end{equation}

Recall that a sequence $\{u_k\} \subset X(\Omega)$ satisfying
\begin{equation}\label{PS seq conds}
\lim_{k\to\infty} \mathcal{F}_{\mu,\lambda}(u_k)=c\in \mathbb{R} \quad \text{and} \quad \lim_{k\to\infty} \mathcal{F}_{\mu,\lambda}^{\prime}(u_k)=0 
\end{equation}
is called a Palais–Smale sequence for $\mathcal{F}_{\mu,\lambda}$ at level $c$. Furthermore, we say that the functional $\mathcal{F}_{\mu, \lambda}$ satisfies the Palais-Smale condition (in short $(\mathrm{PS})_c$) if any sequence satisfying \eqref{PS seq conds} admits a convergent subsequence.

Since the norm of the space $X(\Omega)$ involves the gradient term, we first establish pointwise convergence of the gradient to verify the $(\mathrm{PS})_c$ condition. To this end, we prove the following lemma, which is inspired by \cite[Lemma 2.2]{SFV-2024} and \cite[Lemma 3.2]{MGS2025}.

\begin{lem}\label{grad conv lem}
  Let $\mathcal{F}_{\mu,\lambda}$ be the functional defined as in \eqref{MP functional}. Assume that $\lambda>0$ and $\mu \in (0,\bar{\mu}(\theta))$. If $\left\{u_{k}\right\}$ is a Palais–Smale sequence of $\mathcal{F}_{\mu,\lambda}$, then, there exists $u\in X(\Omega)$ such that up to a subsequence, we have $\nabla u_{k}(x) \rightarrow \nabla u(x)$ a.e. in $\Omega$ as $k \rightarrow \infty$.  
\end{lem}
\begin{proof}
    Since $\left\{u_{k}\right\}$ is a $(\mathrm{PS})_c$ sequence, we have
    \begin{align}\label{boundedness pr 1}
    c+o_k(1) & =\mathcal{F}_{\mu, \lambda}\left(u_{k}\right)-\frac{1}{r}\left\langle \mathcal{F}_{\mu, \lambda}^{\prime}\left(u_{k}\right), u_{k}\right\rangle \nonumber\\ 
    &=\frac{1}{p}\left(\int_{\Omega} |\nabla u_k|^{p}dx+\iint_{\mathbb{R}^{2N}}  \frac{|u_k(x)-u_k(y)|^p}{|x-y|^{N+p s}} d x d y-\mu\int_{\Omega} \frac{|u_k|^{p}}{|x|^{p\theta}} d x\right)-\frac{\lambda}{r}\int_{\Omega}|u_k|^{r}dx\nonumber\\
    &-\frac{1}{r}\left(\int_{\Omega} |\nabla u_k|^{p}dx+\iint_{\mathbb{R}^{2N}}  \frac{|u_k(x)-u_k(y)|^p}{|x-y|^{N+p s}} d x d y-\mu\int_{\Omega} \frac{|u_k|^{p}}{|x|^{p\theta}} d x-\lambda\int_{\Omega}|u_k|^{r}dx\right)\nonumber\\
    &=\left(\frac{1}{p}-\frac{1}{r}\right)\left\|u_{k}\right\|_{X}^{p} 
    -\mu\left(\frac{1}{p}-\frac{1}{r}\right) \int_{\Omega} \frac{|u_{k}|^{p}}{|x|^{p\theta}} dx \\
    & \geq\left(\frac{1}{p}-\frac{1}{r}\right)\left(1-\frac{\mu}{\bar{\mu}}\right)\left\|u_{k}\right\|_{X}^{p}=C\left\|u_{k}\right\|_{X}^{p},\nonumber
    \end{align}
with $C>0$. Hence, we conclude that the sequence $\left\{u_{k}\right\}$ is bounded in $X(\Omega)$. Therefore, from the compact embedding \eqref{Rel-Kon comp emb}, up to a subsequence, still relabeled $\{u_{k}\}$, we have
\begin{align}\label{convergences from bound}
    \begin{array}{ll}
u_{k} \rightharpoonup u \text{ weakly in } X(\Omega), & \nabla u_{k} \rightharpoonup \nabla u \text{ weakly in }\left(L^{p}(\Omega)\right)^{N}, \\
u_{k}(x) \rightarrow u(x) \text{ pointwise a.e. in } \Omega, & \left|u_{k}(x)\right| \leq h(x) \text{ a.e. in } \Omega, \\
u_{k} \rightarrow u \text{ strongly in } L^{q}(\Omega),  
\end{array}
\end{align}
as $k \rightarrow \infty$, where $q \in\left[p, p^{*}\right)$ and $h \in L^{p^{*}}(\Omega)$ with $p^{*}=\frac{p N}{N-p}$. Moreover, from \eqref{PS seq conds}, we have
\begin{equation}\label{bound from ps cond}
\left|\left\langle\mathcal{F}_{\mu, \lambda}^{\prime}\left(u_{k}\right), \phi\right\rangle\right| \leq \epsilon_{k}\|\phi\|_{X} 
\end{equation}
for all $\phi \in X(\Omega)$, and for some $\epsilon_{k}>0$ such that $\epsilon_{k} \rightarrow 0$ as $k\to \infty$. 

Now we fix $n\in \mathbb{N}$ and define the truncation functions $\tau_{n}: \mathbb{R} \rightarrow \mathbb{R}$ as
\begin{equation*}
    \tau_{n}(s)= \begin{cases}s & \text { if }|s| \leq n, \\ n \frac{s}{|s|} & \text { if }|s|>n.\end{cases}
\end{equation*}
Since $\tau_n\left(u_k-u\right)$ is bounded in $X(\Omega)$ and converges to 0 almost everywhere in $\Omega$, it follows (up to a subsequence) that $\tau_n\left(u_k-u\right) \rightharpoonup 0$ weakly in $X(\Omega)$. Therefore, applying H\"older inequality and \eqref{convergences from bound}, we get
\begin{equation}\label{weak conv grad}
    \lim _{k \rightarrow \infty} \int_{\Omega}|\nabla u|^{p-2} \nabla u \nabla\left(\tau_{n}\left(u_{k}-u\right)\right) d x=0,  
\end{equation}
\begin{equation}\label{weak conv gagliardo}
    \lim _{k \rightarrow \infty} \iint_{\mathbb{R}^{2N}} \frac{\mathcal{A}(u(x)-u(y))\left(\tau_{n}\left(u_{k}-u\right)(x)-\tau_{n}\left(u_{k}-u\right)(y)\right)}{|x-y|^{N+p s}} dxdy=0, 
\end{equation}
\begin{equation}\label{weak conv subcrit}
    \lim _{k \rightarrow \infty} \int_{\Omega}|u|^{r-2}u \tau_{n}\left(u_{k}-u\right) d x=0. 
\end{equation}
Moreover, for any measurable set $U \subset \Omega$, using H\"older's inequality and \eqref{convergences from bound}, we obtain 
\begin{align*}
    \left|\int_{U} \frac{|u|^{p-2} u \tau_{n}\left(u_{k}-u\right)}{|x|^{p \theta}}dx\right| &\leq\left(\int_{U} \frac{|u|^{p}}{|x|^{p \theta}}dx\right)^{\frac{p-1}{p}}\left(\int_{U} \frac{\left|\tau_{n}\left(u_{k}-u\right)\right|^{p}}{|x|^{p \theta}}dx\right)^{\frac{1}{p}}\\
    &\leq C\left(\int_{U} \frac{|u|^{p}}{|x|^{p \theta}}dx\right)^{\frac{p-1}{p}}.
\end{align*}
Then, applying Vitali's convergence theorem yields
\begin{equation}\label{weak conv hardy}
\lim _{k \rightarrow \infty} \int_{\Omega} \frac{|u|^{p-2} u \tau_{n}\left(u_{k}-u\right)}{|x|^{p \theta}}dx=0. 
\end{equation}
Consequently, combining \eqref{weak conv grad}-\eqref{weak conv hardy}, we deduce 
\begin{equation*}
      \left\langle\mathcal{F}_{\mu, \lambda}^{\prime}(u), \tau_{n}\left(u_{k}-u\right)\right\rangle=o_{k}(1).
\end{equation*}
Therefore, substituting $v=\tau_{n}\left(u_{k}-u\right)$ in \eqref{bound from ps cond}, we get
\begin{equation*}
\left|\left\langle\mathcal{F}_{\mu, \lambda}^{\prime}\left(u_{k}\right)-\mathcal{F}_{\mu, \lambda}^{\prime}(u), \tau_{n}\left(u_{k}-u\right)\right\rangle\right| \leq \epsilon_{k}\left\|\tau_{n}\left(u_{k}-u\right)\right\|_{X}+o_k(1), 
\end{equation*}
which implies
\begin{align}\label{ineq for grad}
& \int_{\Omega}\left(\left|\nabla u_{k}\right|^{p-2} \nabla u_{k}-|\nabla u|^{p-2} \nabla u\right) \nabla\left(\tau_{n}\left(u_{k}-u\right)\right)dx \nonumber\\
& \quad+\iint_{\mathbb{R}^{2N}}\frac{\left[\mathcal{A}\left(u_{k}(x)-u_{k}(y)\right)-\mathcal{A}(u(x)-u(y))\right]\left(\tau_{n}\left(u_{k}-u\right)(x)\tau_{n}\left(u_{k}-u\right)(y)\right)}{|x-y|^{N+p s}}dxdy \nonumber\\
& \quad \leq \mu\left|\int_{\Omega} \frac{(\left|u_{k}\right|^{p-2} u_{k}-|u|^{p-2} u)\left(\tau_{n}\left(u_{k}-u\right)\right)}{|x|^{p \theta}}dx\right|\\
& \quad +\lambda\left|\int_{\Omega}\left(|u_{k}|^{r-2}u_k-|u|^{r-2}u\right)\left(\tau_{n}\left(u_{k}-u\right)\right)dx\right|+\epsilon_{k}\left\|\tau_{n}\left(u_{k}-u\right)\right\|_{X}+o_k(1).\nonumber
\end{align}
As proved in \cite[Lemma 2.2]{SFV-2024}, we have
\begin{align}\label{gagl part pos}
 \iint_{\mathbb{R}^{2N}}\frac{\left[\mathcal{A}\left(u_{k}(x)-u_{k}(y)\right)-\mathcal{A}(u(x)-u(y))\right]\left(\tau_{n}\left(u_{k}-u\right)(x)\tau_{n}\left(u_{k}-u\right)(y)\right)}{|x-y|^{N+p s}}dxdy \geq 0.
\end{align}
Then, using \eqref{ineq for grad} and \eqref{gagl part pos}, we obtain
\begin{align}\label{bound gor grad 2}
& \int_{\Omega}\left(\left|\nabla u_{k}\right|^{p-2} \nabla u_{k}-|\nabla u|^{p-2} \nabla u\right) \nabla\left(\tau_{n}\left(u_{k}-u\right)\right) dx \nonumber \\
&\quad \quad \quad \quad \quad  \leq \mu\left|\int_{\Omega} \frac{(\left|u_{k}\right|^{p-2} u_{k}-|u|^{p-2} u)\left(\tau_{n}\left(u_{k}-u\right)\right)}{|x|^{p \theta}}dx\right| \\
&\quad \quad \quad \quad \quad  +\lambda\left|\int_{\Omega}\left(|u_{k}|^{r-2}u_k-|u|^{r-2}u\right)\left(\tau_{n}\left(u_{k}-u\right)\right)dx\right| +\epsilon_{k}\left\|\tau_{n}\left(u_{k}-u\right)\right\|_{X}+o_k(1).\nonumber 
\end{align}
Passing to the limit in \eqref{bound gor grad 2} and using \eqref{convergences from bound} and \eqref{weak conv hardy}, we deduce
\begin{align}\label{lim of grad bound}
&\limsup _{k \rightarrow \infty} \int_{\Omega}\left(\left|\nabla u_{k}\right|^{p-2} \nabla u_{k}-\left|\nabla u\right|^{p-2} \nabla u\right)\nabla\left(\tau_{n}\left(u_{k}-u\right)\right)dx \nonumber \\
& \quad \quad \quad \quad \quad   \quad \quad  \leq \mu \limsup _{k \rightarrow \infty}\left|\int_{\Omega} \frac{(\left|u_{k}\right|^{p-2} u_{k}-|u|^{p-2} u)(\tau_{n}\left(u_{k}-u\right))}{|x|^{p \theta}}dx\right| \\
& \quad \quad \quad \quad \quad   \quad \quad  =\mu \limsup _{k \rightarrow \infty}\left|\int_{\Omega} \frac{\left|u_{k}\right|^{p-2} u_{k}\left(\tau_{n}\left(u_{k}-u\right)\right)}{|x|^{p \theta}}dx\right| \nonumber \\
& \quad \quad \quad \quad \quad   \quad \quad  \leq \mu j \limsup _{k \rightarrow \infty}\left(\int_{\Omega} \frac{\left|u_{k}\right|^{p}}{|x|^{p \theta}}dx\right)^{\frac{p-1}{p}}\left(\int_{\Omega} \frac{1}{|x|^{p \theta}}dx\right)^{\frac{1}{p}} \leq n C \nonumber 
\end{align}
for some $C>0$. Now, we set 
\begin{equation*}
    e_{k}(x)=\left[\left|\nabla u_{k}(x)\right|^{p-2} \nabla u_{k}(x)-|\nabla u(x)|^{p-2} \nabla u(x)\right] \nabla\left(u_{k}(x)-u(x)\right).
\end{equation*}
Applying Simon's inequalities from \cite{Simon-1978}, we deduce that $e_{k}(x) \geq 0$ a.e. in $\Omega$. From the boundedness of $\{\nabla u_{k}\}$ in $L^{p}\left(\Omega ; \mathbb{R}^{N}\right)$ and of $\{\left|\nabla u_{k}\right|^{p-2} \nabla u_{k}\}$ in $L^{p^{\prime}}\left(\Omega;\mathbb{R}^{N}\right)$ by \eqref{convergences from bound}, we obtain
\begin{equation}\label{Ap bound}
     0 \leq \int_\Omega e_{k}(x) dx\leq\left\|\left|\nabla u_{k}\right|^{p-2} \nabla u_{k}-\left|\nabla u\right|^{p-2} \nabla u\right\|_{L^{p^{\prime}}(\Omega)}\left\|\nabla u_{k}-\nabla u\right\|_{L^p(\Omega)}\leq C_{0},
\end{equation}   
for some constant $C_{0}$ independent of $k$, where $p^{\prime}=\frac{p}{p-1}$ be the conjugate exponent of $p$.

 For fixed $n, k \in \mathbb{N}$, we split $\Omega$ into
\begin{equation*}
S_{k}^{n}=\left\{x \in \Omega:\left|u_{k}(x)-u(x)\right| \leq n\right\}, \quad G_{k}^{n}=\left\{x \in \Omega:\left|u_{k}(x)-u(x)\right|>n\right\} .
\end{equation*}
Then, taking $\delta \in(0,1)$ and using H\"older inequality and \eqref{Ap bound} together with \eqref{lim of grad bound}, we get
\begin{align*}
\int_{\Omega} e_{k}^{\delta}dx & =\int_{S_{k}^{n}} e_{k}^{\delta}dx+\int_{G_{k}^{n}} e_{k}^{\delta}dx \\
& \leq\left(\int_{S_{k}^{n}} e_{k}dx\right)^{\delta}\left|S_{k}^{n}\right|^{1-\delta}+\left(\int_{G_{k}^{n}} e_{k}dx\right)^{\delta}\left|G_{k}^{n}\right|^{1-\delta} \\
& \leq(n C)^{\delta}\left|S_{k}^{n}\right|^{1-\delta}+(C_0)^{\delta}\left|G_{k}^{n}\right|^{1-\delta}.
\end{align*}
Since $\left|G_{k}^{n}\right| \rightarrow 0$ as $k \rightarrow \infty$, we obtain
\begin{equation*}
0 \leq \limsup _{k \rightarrow \infty} \int_{\Omega} e_{k}^{\delta} d x \leq(n C)^{\delta}|\Omega|^{1-\delta}.
\end{equation*}
Taking $n \rightarrow 0^{+}$, we derive that $e_{k}^{\delta} \rightarrow 0$ in $L^{1}(\Omega)$ as $k \rightarrow \infty$. Hence, passing to a subsequence, we have $e_{k}(x) \rightarrow 0$ a.e. in $\Omega$ as $k \rightarrow \infty$. 
Therefore, by \cite[Lemma 3]{Demengel} we conclude that
\begin{equation*}
    \nabla u_{k} \rightarrow \nabla u \quad \text { a.e. in } \Omega.
\end{equation*}
 This completes the proof.
\end{proof}

\begin{lem}\label{PS cond}
     Let $\lambda>0$ and $\mu \in (0,\bar{\mu}(\theta))$. Then $\mathcal{F}_{\mu,\lambda}$ satisfies the $(\mathrm{PS})_c$ condition for all $c\in \mathbb{R}$.
\end{lem}
\begin{proof}
    Let $\{u_k\}$ be a $(\mathrm{PS})_c$ sequence of the functional $\mathcal{F}_{\mu,\lambda}$. Then, from \eqref{boundedness pr 1}, we have that ${u_k}$ is bounded in $X(\Omega)$, and all the convergences in \eqref{convergences from bound} hold true. Moreover, substituting $v=u_{k}-u$ in \eqref{bound from ps cond}, we obtain
\begin{align}\label{ps cond pr conv 1}
\left\langle\mathcal{F}_{\mu, \lambda}^{\prime}\left(u_{k}\right), u_{k}-u\right\rangle \rightarrow 0 \quad \text{ as } k \rightarrow \infty.
\end{align}

Now from \eqref{convergences from bound} and Lemma \ref{grad conv lem}, we deduce that
$$
\begin{array}{rlr}
\left|\nabla u_{k}(x)\right|^{p-2} \nabla u_{k}(x) & \rightarrow|\nabla u(x)|^{p-2} \nabla u(x) & \text { pointwise a.e. in } \Omega, \\
\frac{\left|u_{k}(x)\right|^{p-2} u_{k}(x)}{|x|^{p \theta/p^{\prime}}} & \rightarrow \frac{|u(x)|^{p-2} u(x)}{|x|^{p \theta/p^{\prime}}} & \text { pointwise a.e. in } \Omega. \\
\end{array}
$$
 Moreover, the sequences $\left\{\left|\nabla u_{k}\right|^{p-2} \nabla u_{k}\right\}$ and $\left\{\frac{\left|u_{k}\right|^{p-2} u_{k}}{|x|^{p \theta/p^{\prime}}}\right\}$ are bounded in $L^{p^{\prime}}(\Omega)$, and hence, converge weakly in $L^{p^{\prime}}(\Omega)$, up to a subsequence.
 
 Let $\mathcal{A}(t)=|t|^{p-2}t$. Then as shown in the proof of \cite[Lemma 2.4]{CP-2016}, we have that 
 \begin{equation*}
     \frac{\mathcal{A}\left(u_{k}(x)-u_{k}(y)\right)}{|x-y|^{(N+s p)/p^{\prime}}}  \rightarrow \frac{\mathcal{A}(u(x)-u(y))}{|x-y|^{(N+s p)/p^{\prime}}}  \text { pointwise a.e. in } \mathbb{R}^{2 N},
 \end{equation*}
 and the sequence $\left\{\frac{\mathcal{A}\left(u_{k}(x)-u_{k}(y)\right)}{|x-y|^{(N+s p)/p^{\prime}}}\right\}$ is bounded in $L^{p^{\prime}}\left(\mathbb{R}^{2 N}\right)$. Since weak and pointwise limits coincide, passing to the limit as $k\to\infty$, we obtain, up to a subsequence, that
 \begin{align}\label{weak cons for PS}
\int_{\Omega}\left|\nabla u_{k}\right|^{p-2} \nabla u_{k} \nabla \varphi dx& \rightarrow \int_{\Omega}\left|\nabla u\right|^{p-2} \nabla u \nabla \varphi dx, \nonumber\\
\int_{\Omega} \frac{\left|u_{k}\right|^{p-2} u_{k} \varphi}{|x|^{p \theta}} d x & \rightarrow \int_{\Omega} \frac{\left|u\right|^{p-2} u \varphi}{|x|^{p \theta}} d x,  \\
 \iint_{\mathbb{R}^{2N}} \frac{\mathcal{A}\left(u_{k}(x)-u_{k}(y)\right)(\varphi(x)-\varphi(y))}{|x-y|^{N+s p}}dxdy & \rightarrow \iint_{\mathbb{R}^{2N}}  \frac{\mathcal{A}\left(u(x)-u(y)\right)(\varphi(x)-\varphi(y))}{|x-y|^{N+s p}}dxdy,\nonumber
\end{align} 
 for any $\varphi \in X(\Omega)$, since $\nabla \varphi\in L^{p}(\Omega)$, $\varphi/|x|^\theta  \in L^{p}(\Omega)$,
and $\frac{|\varphi(x)-\varphi(y)|}{|x-y|^{(N+p s) / p}} \in L^{p}\left(\mathbb{R}^{2 N}\right)$.

 In particular, taking $\varphi=u$ in \eqref{weak cons for PS}, we have as $k \rightarrow \infty$ that
\begin{align}\label{nec cons for PS}
\int_{\Omega}\left|\nabla u_{k}\right|^{p-2} \nabla u_{k} \nabla u dx& \rightarrow \int_{\Omega}|\nabla u|^{p}dx, \nonumber\\
\int_{\Omega} \frac{\left|u_{k}\right|^{p-2} u_{k} u}{|x|^{p \theta}} d x & \rightarrow \int_{\Omega} \frac{|u|^{p}}{|x|^{p \theta}} d x,  \\
 \iint_{\mathbb{R}^{2N}} \frac{\mathcal{A}\left(u_{k}(x)-u_{k}(y)\right)(u(x)-u(y))}{|x-y|^{N+s p}}dxdy & \rightarrow \iint_{\mathbb{R}^{2N}}  \frac{|u(x)-u(y)|^{p}}{|x-y|^{N+s p}}dxdy. \nonumber
\end{align}

Furthermore, by \eqref{convergences from bound}, Lemma \ref{grad conv lem}, and the Brezis-Lieb lemma \cite{B-L}, we have
\begin{align}\label{BL cons for PS}
\int_{\Omega}\left|\nabla\left(u_{k}-u\right)\right|^{p}dx & =\int_{\Omega}\left|\nabla u_{k}\right|^{p}dx-\int_{\Omega}|\nabla u|^{p}dx+o_{k}(1), \nonumber\\
{\left[u_{k}-u\right]_{s,p}^{p} } & =\left[u_{k}\right]_{s,p}^{p}-[u]_{s,p}^{p}+o_{k}(1),  \\
\int_{\Omega} \frac{\left|u_{k}-u\right|^{p}}{|x|^{p \theta}}dx & =\int_{\Omega} \frac{\left|u_{k}\right|^{p}}{|x|^{p \theta}}dx-\int_{\Omega} \frac{|u|^{p}}{|x|^{p \theta}}dx+o_{k}(1).\nonumber
\end{align}

Thus, using \eqref{convergences from bound}, \eqref{ps cond pr conv 1}, \eqref{nec cons for PS}, \eqref{BL cons for PS}, and the mixed interpolated Hardy inequality \eqref{mixed Hardy use}, we obtain
\begin{align*}
&o_k(1)=\left\langle\mathcal{F}_{\mu, \lambda}^{\prime}\left(u_{k}\right), u_{k}-u\right\rangle \\
&\quad =\int_{\Omega}\left|\nabla u_{k}\right|^{p}dx-\int_{\Omega}|\nabla u_k|^{p-2}u_k udx+\left[u_{k}\right]_{s,p}^{p}-\iint_{\mathbb{R}^{2N}} \frac{\mathcal{A}\left(u_{k}(x)-u_{k}(y)\right)(u(x)-u(y))}{|x-y|^{N+s p}}dxdy\\
&\quad +\mu\int_{\Omega} \frac{\left|u_{k}\right|^{p}}{|x|^{p \theta}}dx-\mu\int_{\Omega} \frac{|\nabla u_k|^{p-2}u_k u}{|x|^{p \theta}}dx+\int_{\Omega}\left|u_{k}\right|^{r}dx-\int_{\Omega}| u_k|^{r-2}u_k udx\\
&\quad = \int_{\Omega}\left|\nabla u_{k}\right|^{p}dx-\int_{\Omega}|\nabla u|^{p}dx+\left[u_{k}\right]_{s,p}^{p}-[u]_{s,p}^{p}+\mu\int_{\Omega} \frac{\left|u_{k}\right|^{p}}{|x|^{p \theta}}dx-\mu\int_{\Omega} \frac{|u|^{p}}{|x|^{p \theta}}dx\\
&\quad +\int_{\Omega}\left|u_{k}\right|^{r}dx-\int_{\Omega}\left|u\right|^{r}dx+o_{k}(1) \\
&\quad = \int_{\Omega}\left|\nabla\left(u_{k}-u\right)\right|^{p}dx+\iint_{\mathbb{R}^{2N}} \frac{\left|\left(u_{k}-u\right)(x)-\left(u_{k}-u\right)(y)\right|^{p}}{|x-y|^{N+p s}} d x d y\\
&\quad -\mu \int_{\Omega} \frac{\left|u_{k}-u\right|^{p}}{|x|^{p \theta}}dx +o_{k}(1)\geq \left(1-\frac{\mu}{\bar{\mu}(\theta)}\right)  \|u_k-u\|_{X}+o_k(1).
\end{align*}
This implies $u_{k}\to u$ strongly in $X(\Omega)$ as $k \rightarrow \infty$, completing the proof.
\end{proof}

Now we show that the functional $\mathcal{F}_{\mu, \lambda}$ defined in \eqref{MP functional} verifies the assumptions of Theorem \ref{MP thm}.
\begin{lem}\label{geometry}
    Let $\Omega$ be a bounded domain in $\mathbb{R}^N$. Assume that $\lambda>0$ and $\mu \in (0,\bar{\mu}(\theta))$. Then there exist positive constants $\rho, \alpha>0$ such that 
    
    a) $\mathcal{F}_{\mu, \lambda}(u) \geq \alpha$ for any $u \in X(\Omega)$ with $\|u\|_{X}=\rho$; 
    
    b) There exists $v \in X(\Omega)$ positive such that $\|v\|_{X} > \rho$ and  $\mathcal{F}_{\mu, \lambda}(v) < \alpha$.
\end{lem} 
\begin{proof}
    a)  
Since $p<r<p^*,$ applying the embedding \eqref{Sobolev emb}, and the mixed Hardy inequality \eqref{mixed Hardy use} we get
\begin{align*}
\mathcal{F}_{\mu, \lambda}(u) & =\frac{1}{p} \left(\int_{\Omega}\left|\nabla u\right|^{p} dx+\iint_{\mathbb{R}^{2N}}  \frac{|u(x)-u(y)|^{p}}{|x-y|^{N+s p}}dxdy\right)-\frac{\mu}{p} \int_{\Omega} \frac{|u|^{p}}{|x|^{p\theta}} dx-\frac{\lambda}{r} \int_{\Omega} |u|^{r} dx \\
& {\geq} \frac{1}{p}\|u\|_{X}^{p}-\frac{\mu}{p \bar{\mu}}\|u\|_{X}^{p}-\frac{\lambda}{r} \int_{\Omega} |u|^{r} d x\\
& {\geq}\frac{1}{p}\left(1-\frac{\mu}{\bar{\mu}}\right)\|u\|_{X}^{p}-\frac{\lambda C_2}{r}\|u\|_{X}^{r} \\
& {=} \frac{C_1}{p}\|u\|_{X}^{p}-\frac{\lambda C_2}{r}\|u\|_{X}^{r},
\end{align*}
where $C_1, C_2$ are positive constants. As $r>p$, there exists small enough $\rho>0$ such that
$$\alpha:=\frac{C_{1} \rho^{p}}{p}-\lambda\frac{C_{2} \rho^{r}}{r}>0.$$

\noindent Therefore, we have $\mathcal{F}_{\mu, \lambda}(u) \geq \alpha$ for any $u \in X(\Omega)$ with $\|u\|_{X}=\rho$.

b) Fix $v_{0} \in X(\Omega)$ positive such that $\|v_{0}\|_{X}=1$ and take $t>0$. Then we have
\begin{align*}
\mathcal{F}_{\mu, \lambda}(t v_{0}) & =\frac{1}{p}\|t v_{0}\|_{X}^{p}-\frac{\mu}{p} \int_{\Omega} \frac{|t v_{0}|^{p}}{|x|^{p\theta}} dx -\frac{\lambda}{r} \int_{\Omega}|t v_{0}|^{r} d x \\
& \leq \frac{1}{p}\|t v_{0}\|_{X}^{p}-\frac{\lambda}{r} \int_{\Omega}|t v_{0}|^{r} d x  \\
& =\frac{t^{p}}{p}-\frac{\lambda t^{r}}{r} \int_{\Omega} |v_{0}|^{r} dx.
\end{align*}
Again using  $r>p$, we get $\mathcal{F}_{\mu, \lambda}(t v_{0}) \rightarrow-\infty$ as $t \rightarrow+\infty$. Therefore, for sufficiently large $t_0,$ with $v=t_{0} v_{0}$, it follows that $\|v\|_{X}>\rho$ and $\mathcal{F}_{\mu, \lambda}(v) < \alpha$.  
\end{proof}

\begin{proof}[Proof of Theorem \ref{main result thm 1}]
As we proved in Lemmas \ref{PS cond} and \ref{geometry}, the functional $\mathcal{F}_{\mu,\lambda}(u)$ satisfies the geometry of the mountain pass lemma and satisfies the $(\mathrm{PS})_{c}$ condition for all $c\in \mathbb{R}$. Moreover, we have $\mathcal{F}_{\mu,\lambda}(0)=0$. Now we take $v$ as in Lemma \ref{geometry} and define 
$$
c:=\inf_{\gamma \in \Gamma} \sup _{t \in[0,1]} \mathcal{F}_{\mu, \lambda}(\gamma(t)),
$$
with $\Gamma:=\left\{\gamma \in C\left([0,1], X(\Omega)\right): \gamma(0)=0, \gamma(1)=v\right\}.$ Then, by Theorem \ref{MP thm}, $c$ is a critical value of $\mathcal{F}_{\mu, \lambda}(u)$. Consequently, there exists $u \in X(\Omega)$ satisfying
$$\mathcal{F}_{\mu, \lambda}(u)=c\geq \rho > 0, $$
which is a nontrivial weak solution to the problem \eqref{main problem 0.1}.
\end{proof}

\section*{Conflict of interest statement}
On behalf of all authors, the corresponding author states that there is no conflict of interest.

\section*{Data availability statement}
Data sharing is not applicable to this article as no datasets were generated or analyzed during the current study.

\section*{Acknowledgement}
YA is supported by the Bolashak Government Scholarship of the Republic of Kazakhstan.  SG acknowledges financial support under the ARG-MATRICS grant from ANRF, India, (Grant number: ANRF/ARGM/2025/001570/MTR). YA and MR are supported by the Methusalem program of the Ghent University Special Research Fund (BOF) (Grant number 01M01021). MR is also supported by UKRI EPSRC grant (Grant number UKRI3645) and the FWO Senior Research Grant G011522N. This work was completed while VK and SG were visiting the Ghent Analysis \& PDE Center at Ghent University. They gratefully acknowledge the financial support and excellent research facilities provided by the center.


\end{document}